\documentclass{m2an}

\usepackage{amsmath,amsfonts,amscd,amssymb,bm,mathtools}
\usepackage{enumerate}
\usepackage{tikz}
\usepackage{array}
\usepackage{booktabs}
\newcommand{\reals}{\mathbb{R}}
\newcommand{\N}{\mathbb{N}}

\newcommand{\cA}{{\mathcal{A}}}
\newcommand{\cC}{{\mathcal{C}}}

\newcommand{\cR}{\mathcal{R}}
\newcommand{\cS}{\mathcal{S}}

\newcommand{\R}{\mathbb{R}}

\newcommand{\cB}{\mathcal{B}}
\newcommand{\bu}{\mathbf u}
\newcommand{\bU}{\mathbf U}

\newcommand{\bv}{\mathbf v}

\newcommand{\bq}{\mathbf q}

\newcommand{\bw}{\mathbf w}

\newcommand{\bF}{\mathbf F}
\newcommand{\bD}{\mathbf D}
\newcommand{\be}{\mathbf e}

\newcommand{\bftau}{\boldsymbol{\tau}}
\newcommand{\bfsigma}{\boldsymbol{\sigma}}
\newcommand{\bn}{\boldsymbol{n}}

\newcommand{\bE}{\mathbf{E}}

\newcommand{\grad}{\nabla}
\newcommand{\Om}{\Omega}

\newcommand{\eps}{\epsilon}
\def\ltt{\lesssim}

\newcommand{\dt}{{\Delta t}}
\newcommand{\Th}{{\mathcal{T}_h}}

\DeclareMathOperator{\dist}{dist}
\DeclareMathOperator{\signum}{sgn}
\DeclareMathOperator{\sdist}{Sdist}

\theoremstyle{plain}
\newtheorem{theorem}{Theorem}[section]
\newtheorem{lemma}[theorem]{Lemma}
\newtheorem{proposition}[theorem]{Proposition}
\newtheorem{definition}{Definition}
\theoremstyle{remark}
\newtheorem{remark}{Remark}[section]

\begin{document}

\title{Analysis of a diffuse interface method for the Stokes-Darcy coupled problem}

\author{Martina Buka{\v{c}}}\address{Department of Applied and Computational Mathematics and Statistics, University of Notre Dame, South Bend, IN 46556, USA. Email: \texttt{mbukac@nd.edu}.}

\author{Boris Muha}\address{Department of Mathematics, Faculty of Science, University of Zagreb, Croatia. Email: \texttt{borism@math.hr}.}

\author{Abner J. Salgado}\address{Department of Mathematics, University of Tennessee, Knoxville TN 37996 USA. Email: \texttt{asalgad1@utk.edu}.}




\begin{abstract}
We consider the interaction between a free flowing fluid and a porous medium flow, where the free flowing fluid is described using the time dependent Stokes equations, and the porous medium flow is described using  Darcy's law in the primal formulation. To solve this problem numerically, we use a diffuse interface approach, where the weak form of the coupled problem is written on an extended domain which contains both Stokes and Darcy regions. This is achieved  using a phase-field function which equals one in the Stokes region and zero in the Darcy region, and smoothly transitions between these two values on a diffuse region of width $\mathcal{O}(\epsilon)$ around the Stokes-Darcy interface. We prove  convergence of the diffuse interface formulation to the standard, sharp interface formulation, and derive rates of convergence. This is performed by deriving a priori error estimates for discretizations of the diffuse interface method, and by analyzing the modeling error of the diffuse interface approach at the continuous level.  The convergence rates are also shown computationally in a  numerical example.
\end{abstract}

\maketitle

\section{Introduction}
The interaction between a free flowing fluid and a porous medium flow, commonly formulated as a Stokes-Darcy  coupled system, has been used to describe problems arising in many applications, including environmental sciences, hydrology, petroleum engineering and biomedical engineering. Hence, the development of numerical methods for  Stokes-Darcy problems has been an active area of research.  Most  existing numerical methods are based on a sharp interface approach, in the sense that the interface between the Stokes and Darcy regions is parametrized using an exact specification of its geometry and location, and the nodes in the computational mesh align with the interface. They include both monolithic and partitioned numerical methods. Some of the recently developed monolithic schemes include a two-grid method with backtracking for the stationary monolithic Stokes-Darcy problem  proposed in~\cite{dutwo} and a mortar multiscale finite element method presented  in~\cite{girault2014mortar}. We also mention  the work in~\cite{burman2007unified,yu2020nitsche}  based on the Nitsche's penalty method.  Non-iterative partitioned schemes  based on various time-discretization strategies were developed  in~\cite{layton2012long,kubacki2015analysis,layton2013analysis,layton2012stability,chen2013efficient,gunzburger2018stokes}. A third-order in time implicit-explicit algorithm based on the combination of the Adams–Moulton and the Adams--Bashforth scheme  was proposed in~\cite{chen2016efficient}, and iterative domain decomposition methods based on generalized Robin coupling conditions  were derived  in~\cite{discacciati2007robin,discacciati2018optimized}.

While the sharp interface methods are widely used, the explicit interface parametrization may be difficult to obtain in case of complex geometries.  The exact location is sometimes  not known,  or the  geometry is complicated, making a proper approximation of the integrals error-prone and difficult to automate. This often occurs when geometries are obtained implicitly using imaging data, commonly used in patient-specific biomedical simulations. 
Hence, as one alternative to  sharp interface approaches, diffuse interface methods have been introduced~\cite{bueno2006spectral,ratz2006pde,li2009solving,lervag2009analysis,burger2017analysis,anderson1998diffuse,du2020phase}. They are also known as phase-field, or diffuse domain methods. While they have features in common with the level set method~\cite{chen1991uniqueness,osher1988fronts,osher2003level,pacquaut2012combining}, they are  fundamentally different since the  level set method tracks the exact sharp interface
without introducing any diffuse layers~\cite{du2020phase}.  Other conceptually similar approaches include the fictitious domain method with a spread interface~\cite{ramiere2007fictitious, ramiere2007general}, the immersed boundary method with
interface forcing functions based on Dirac distributions~\cite{peskin2002immersed,griffith2005order} and the fat boundary method~\cite{maury2001fat,bertoluzza2011analysis}, among others.
The diffuse interface method has received strong attention from the applications point of view~\cite{stoter2017diffuse,teigen2009diffuse,elliott2011numerical,aland2010two,teigen2011diffuse,miehe2010phase,borden2012phase,mikelic2015phase,gomez2013three,liu2015liquid,saylor2016diffuse}. However,  many difficult questions remain to be theoretically addressed both
at the {continuous} and the discrete level. One of the fundamental mathematical questions, whether the diffuse interface  converges to a sharp interface  when the width of the diffuse interface tends to zero, remains an open problem for
many phase-field models. 

Theoretically, the diffuse  interface approach has been widely studied for elliptic problems and two-phase flow problems. 
An analysis of the diffuse interface method for elliptic problems has been provided in~\cite{burger2017analysis,burger2015diffuse,schlottbom2016error,franz2012note,li2009solving}. An elliptic problem was also considered in~\cite{nguyen2018diffuse}, where  diffuse formulations of Nitsche's method for imposing Dirichlet boundary conditions on phase-field approximations of sharp domains were studied. A Cahn-Larch\'e phase-field model approximating an elasticity sharp interface problem has been considered in~\cite{abels2015sharp}. Advection--diffusion on evolving diffuse interfaces has been studied in~\cite{elliott2011numerical,teigen2009diffuse}. We also mention the work in~\cite{gao2018decoupled} where the interaction between the two-phase free flow and two-phase porous media flow was considered, but the diffuse interface method was used to separate different phases in each region, while the Stokes-Darcy interface was captured by a mesh aligned with the interface and the two problems were decoupled  using a classical approach.  Diffuse interface models for two-phase flows have been extensively analyzed, see, e.g.,~\cite{abels2009diffuse,abels2014sharp,feireisl2010analysis,abels2017sharp,guo2021diffuse,ray2021discontinuous,yang2019diffuse}. However, only recently a rigorous sharp interface limit convergence result with convergence rates in strong norms has been proved~\cite{abels2018sharp}. 
The main challenges in the analysis of the diffuse interface problems involve proving the convergence of the diffuse interface to the sharp
interface, and the estimation of the modeling error, which can be quite difficult to analyze and
remain open for many problems.

In this work, we consider the  fluid--porous medium interaction  described by the time--dependent Stokes--Darcy coupled problem. We discretize the problem in time using the backward Euler method, and in space using the finite element method. The focus of this paper is  the convergence analysis of the diffuse interface to the sharp interface as the width of the diffuse layer goes to zero. This is performed in two steps. First, we 
derive error estimates and prove the convergence for the finite element approximation of the diffuse interface method. Then, we analyze the convergence of the continuous diffuse interface formulation to the continuous sharp interface formulation, obtaining the convergence rates with respect to the width of the diffuse layer. These rates of convergence are also explored numerically.


In the diffuse interface approach, the integrals contain some form of the phase-field function, which can be thought of being a weight and naturally introduces the framework of weighted Sobolev spaces. 
Hence, the analysis in this work relies on the results  from~\cite{burger2017analysis} on the convergence of the diffuse integrals, and on a trace lemma, which states that the trace operator for weighted Sobolev spaces is uniformly bounded for certain perturbations of the domain. A variety of results for weighted spaces with  Muckenhoupt weights~\cite{MR775568} are used as well. The main challenge comes from the fact that the weights used in~\cite{burger2017analysis} do not belong to a  Muckenhoupt  class, and that the weights from the Muckenhoupt  class used in this work ($A_2$, as defined in Section~\ref{sec:Notation}) are not necessarily Lipschitz. 
Therefore, the phase-field functions used in our analysis had to be carefully constructed. In addition, due to the setting of the problem, a special attention had to be given to extend both the continuous and the discrete inf-sup conditions in order to construct the pressure in the correct space.

The rest of this paper is organized as follows. Section~\ref{sec:Notation} introduces notation and some background information that will prove useful in our developments. The mathematical models, together with their well-posedness, are described in Section~\ref{sec:Models}. The numerical approximation of the diffuse interface formulation is developed in Section~\ref{sec:Discretization}, where we develop and analyze a fully discrete scheme for this problem. The so-called modeling error is the topic of Section~\ref{Sec_ModellinError} where we derive modeling error estimates under two scenarios: the case when the phase-field function remains strictly positive, and that when it vanishes. Finally, several numerical illustrations that not only show our proved error estimates, but also demonstrate the flexibility of our diffuse interface approach are presented in Section~\ref{Sec:numerics}. 

\section{Notation and preliminaries}\label{sec:Notation} 

Let $\Omega \subset \mathbb{R}^d$, with $d=2,3$ be a bounded domain with Lipschitz boundary. When dealing with discretization, we shall also assume that this is a polytope. This will be used to denote the domain where the fluid and porous medium interaction takes place. Let $a\lesssim (\gtrsim) b$ denote that there exists a positive constant $C$, independent of discretization or phase-field parameters, such that
$a \leq (\geq) C b$.

We will mostly adhere to standard notation and terminology with regards to function spaces and their properties. Vector and matrix valued functions will be denoted using boldface. For a Banach space $X$ we denote its dual by $X^*$, and the duality pairing is $\langle \cdot,\cdot \rangle_{X}$. We may depart from standard notation in that we will make use of weighted spaces and their properties.

We say that an almost everywhere positive function $\omega \in L^1_{loc}(\mathbb{R}^d)$ is a weight. Every weight, $\omega$, induces a measure with density $\omega dx$, over the Borel sets of $\mathbb{R}^d$ which, for simplicity, will also be denoted by $\omega$. In other words, for $E \subset \mathbb{R}^d$ a Borel set, we let $\omega(E) = \int_E \omega$.

For $r\in (1, \infty)$, $\omega$ a weight, and $D \subset \mathbb{R}^d$ a bounded domain, we define weighted Lebesgue spaces and their norms
\[
  L^r (D, \omega) = \left\{ \psi : D \rightarrow \mathbb{R} \; : \; |\psi|^r \omega \in L^1(D) \right\},
  \qquad
  \| \psi \|_{L^r(D,\omega)}^r = \int_D |\psi|^r \omega.
\]
These spaces are complete. Associated with weighted $L^r-$spaces, we define the weighted Sobolev spaces:
\[
  W^{k,r}  (D, \omega)= \{ \psi \in L^r (D, \omega) \; : \; D^{\alpha} \psi \in L^r(D, \omega), \; \forall \alpha \in \mathbb{N}_0^d : |\alpha| \leq k\},
  \quad
    \|\phi\|_{W^{k,r}(D,\omega)}^r=\sum_{|\alpha|\leq k}\|D^{\alpha}\psi\|^r_{L^r(D,\omega)}.
\]
As usual, we set $H^k(D,\omega) = W^{k,2}(D,\omega)$ for any $k \in \mathbb{N}_0$. In general, one needs caution when defining weighted Sobolev spaces for general weights. For instance, for some weights, we may not have that $L^r(D,\omega) \subset L^1_{loc}(D)$, so that defining weak derivatives may be an issue. Even when this is possible, the ensuing spaces may not be complete \cite{MR775568}. For this reason, we shall always assume that the weight $\omega$ belongs to the Muckenhoupt class $A_r$ \cite{nochetto2016piecewise}. We will only make use of the case $r=2$, so we only define the $A_2$--class. We say that $\omega \in A_2$ if
\[
  [\omega]_{A_2} = \sup_B \left( \frac1{|B|} \int_B \omega \right)\left( \frac1{|B|} \int_B \frac1{\omega} \right) < \infty,
\]
where the supremum is taken over all balls in $\mathbb{R}^d$. If $\omega \in A_2$, then $H^k(D,\omega)$ enjoys most of the properties that its classical unweighted counterpart has. For instance, it is Hilbert and separable; see \cite{MR775568,nochetto2016piecewise}. Below we list some properties of $H^1(D,\omega)$, with $\omega \in A_2$, that shall be useful to us in our developments.
\begin{enumerate}[$\bullet$]
  \item \emph{Traces}: Let $\Gamma \subset \partial D$ have positive and finite $(d-1)$--Hausdorff measure. Then, there is a continuous trace operator $\gamma_\Gamma \colon H^1(D,\omega) \to L^1(\Gamma)$; see \cite{MR3227013,MR3479981}. Recall that this, in particular, implies that subspaces of $H^1(D,\omega)$ consisting of functions that vanish on a piece of its boundary are closed, and hence Hilbert themselves.
  
  In light of what is said above, it is legitimate to define $H^1_0(D,\omega)$ as the subset of $H^1(D,\omega)$ consisting of functions whose trace, onto $\partial D$, vanishes.
  
  \item \emph{Poincar\'e--Friedrichs--type inequalities}: Let $\Gamma$ be as above. There is a constant $C_P>0$ such that, for all $v \in H^1(D,\omega)$, we have
  \begin{equation}
  \label{eq:WeightedPoincare}
    \| v \|_{H^1(D,\omega)}^2 \leq C_P \left( \left| \int_\Gamma v \right| + \| \nabla v \|_{L^2(D,\omega)^d} \right)^2;
  \end{equation}
  see \cite{nochetto2016piecewise} and references therein.
  
  \item \emph{Bogovski\u\i{} operator}: There is $\beta >0$ such that, for all $\psi \in L^2(D,\omega^{-1})$ with $\int_D \psi = 0$, we have 
  \begin{equation}
  \label{eq:Bogovskii}
    \beta \| \psi \|_{L^2(D,\omega^{-1})} \leq \sup_{\boldsymbol0 \neq \bv \in H^1_0(D,\omega)^d} \frac{\int_D \psi \nabla\cdot \bv}{ \|\nabla \bv \|_{L^2(D,\omega)^{d \times d}}},
  \end{equation}
  see \cite[Lemma 6.1]{duran2020stability}.

  \item \emph{Korn's inequality}: For $\bv \in H^1(D,\omega)^d$ we denote by $\bD(\bv) = \tfrac12 \left( \nabla \bv + \nabla \bv \right)^T$ its symmetric gradient. There is a constant $\bar{C}_K>0$ such that
  \begin{equation}
  \label{eq:wKorn1}
    \| \nabla \bv \|_{L^2(D,\omega)^{d \times d}} \leq \bar{C}_K \left( \| \bv \|_{L^2(D,\omega)^d} + \| \bD(\bv) \|_{L^2(D,\omega)^{d \times d}} \right), \quad \forall \bv \in H^1(D,\omega)^d.
  \end{equation}
  One can easily then mimic classical arguments, for instance \cite[Theorem 7.3.2]{MR969367}, to conclude that there is a constant $C_K>0$ such that for all $\bv \in H^1(D,\omega)^d$ that vanish on $\Gamma$, we have
  \begin{equation}
  \label{eq:weightedKorn}
    \| \bv \|_{H^1(D,\omega)^d}^2 \leq C_K \| \bD(\bv) \|_{L^2(D,\omega)^{d \times d}}^2.
  \end{equation}
\end{enumerate}
The constants in all the statements above depend on $\omega$ only through $[\omega]_{A_2}$.

Finally, we mention a prototypical example of a weight that belongs to $A_2$. It has been shown in \cite{MR3215609} and \cite[Lemma 2.3(vi)]{MR1601373} that if $\Gamma \subset \mathbb{R}^d$ is an embedded hypersurface (i.e., a curve if $d=2$ and a surface if $d=3$), then
\begin{equation}
\label{eq:distisA2}
  \omega(x) = \dist_\Gamma(x)^\alpha \in A_2,
\end{equation}
if and only if $\alpha \in (-1,1)$. Here $\dist_\Gamma$ is the distance function to $\Gamma$. If $\Gamma \subset \partial D$, \cite{MR1258430} has provided a finer characterization of the trace operator $\gamma_\Gamma$; namely,
\begin{equation}
\label{eq:TraceIsL2}
  \gamma_\Gamma \left( H^1(D,\dist_\Gamma^\alpha) \right) \hookrightarrow L^2(\Gamma).
\end{equation}

\section{The mathematical models}\label{sec:Models}

Let us now present our models of interest, and discuss their most important properties. We recall that $\Omega \subset \mathbb{R}^d$ denotes an open bounded domain with Lipschitz boundary. This is the domain where the interaction between fluid and porous medium shall take place. $T>0$ is the final time.

\subsection{The sharp interface model}\label{sub:SharpModel}

We begin by describing the sharp interface model, and studying its well-posedness.
Let  $\Omega_F$ denote the fluid domain and $\Omega_D$ denote the porous medium domain. We assume that $\Omega_F$ and $\Omega_D$ are bounded domains in $\mathbb{R}^d$, with $d=2,3$, that have Lipschitz boundary. Moreover, $\Omega_F \cap \Omega_D = \emptyset$, $\bar\Omega = \bar\Omega_F \cup \bar\Omega_D$, and $\Gamma$ is the common boundary between these domains, i.e., $\partial \bar{\Omega}_F \cap \partial \bar{\Omega}_D=\Gamma$. We assume that $\Gamma$ is sufficiently smooth, and that it has positive and finite $(d-1)$--Hausdorff measure $\mathcal{H}_{d-1}(\Gamma)$.

To model the fluid flow, we use the time--dependent Stokes equations, given as follows:
\begin{align}
  \label{NS}
    \rho \partial_t\bu =\nabla\cdot\bfsigma(\bu, \pi)+\rho \bF, \qquad{\rm in}\; \Omega_F \times (0,T), \\
  \label{Incomp}
    \nabla\cdot\bu=0,\qquad{\rm in}\;\Omega_F\times (0,T),
\end{align}
where $\bf{u}$ is the fluid velocity, $\rho>0$ is the, constant, fluid density, $\bfsigma(\bu,\pi)=2\mu \bD(\bu)-\pi \mathbb{I}$ is the fluid Cauchy stress tensor, $\mu$ is the fluid viscosity, $\pi$ denotes the pressure and $\bF$ is the density of volumetric forces.

The porous medium flow is governed by  Darcy's law, written in the primal formulation as:
\begin{align}
\label{DarcyPrimal}
  c_0\partial_t p-\nabla\cdot(\boldsymbol\kappa\nabla p)=g, \qquad{\rm in}\;\Omega_D\times (0,T),
\end{align}
where $p$ is the Darcy pressure, $c_0$ is the mass storativity and $\boldsymbol \kappa$  is the hydraulic conductivity tensor. 
We assume that $\boldsymbol \kappa$ is uniformly bounded and positive definite, so that
$
0< k_* \leq \lambda \leq k^*,
$
where $ \lambda  \in \lambda (\boldsymbol \kappa)$, and $\lambda(\boldsymbol \kappa)$ is the spectrum of $\boldsymbol \kappa$. We note that the porous medium flow velocity can be computed from $p$ as $\bq = - \boldsymbol\kappa \nabla p$.

\noindent 
\textbf{Coupling conditions:} The Stokes and Darcy problems are coupled using the following interface conditions:
\begin{enumerate}[$\bullet$]
  \item The conservation of mass:
  \[
    \bu\cdot\bn=- \boldsymbol\kappa \nabla p \cdot\bn, \qquad {\rm{on}}\; \Gamma \times(0,T),    
  \]
  where $\bn$ is the unit normal to $\Gamma$ which points towards $\Omega_D$.

  \item The Beavers-Joseph-Saffman condition:
  \[
    \alpha_{BJ}\bu\cdot\bftau_{i}+\bfsigma(\bu,\pi) \bn \cdot\bftau_{i}=0, \qquad  i=1,2, \dots, d-1, \qquad{\rm on}\; \Gamma  \times(0,T),    
  \]
  where, for every point in $x\in \Gamma$, $\boldsymbol{\tau}_{i}$, $i=1,\ldots,d-1$, is an orthonormal basis for the tangent space $T_x(\Gamma)$; and $\alpha_{BJ} >0$ is the Beavers-Joseph-Saffman-Jones coefficient.

  \item The balance of pressure:
  \[
    -\bfsigma(\bu,\pi)\bn \cdot\bn =p, \qquad{\rm on}\; \Gamma \times (0,T).    
  \]
\end{enumerate}

%

%
%

\noindent
\textbf{Boundary conditions:}
We split the boundaries as: $\partial\Omega_F=\Gamma\cup\Gamma_F^1\cup\Gamma_F^2$, $\partial\Omega_D=\Gamma\cup\Gamma_D^1\cup\Gamma_D^2$, and prescribe the following boundary conditions:
\begin{align}
  \bu&=\boldsymbol0\quad {\rm on}\quad \Gamma_F^1\times (0,T),
  &
 \label{StokesBC}
  \bfsigma(\bu,\pi) \bn&=\boldsymbol0\quad{\rm on}\quad \Gamma_F^2\times (0,T),
  \\
  p&=0\quad {\rm on}\quad \Gamma_D^1\times (0,T),
  &
 \label{DarcyBC}
  \partial_n p&=0 \quad {\rm on}\quad \Gamma_D^2\times (0,T).
\end{align}

\begin{remark}[Boundary conditions]
In this paper, for simplicity, we consider only homogeneous boundary conditions. However, inhomogeneous boundary conditions can be easily included in the analysis by using the appropriate extension operators to homogenize the boundary conditions.
\end{remark}

\noindent
\textbf{Initial conditions:} Finally, we supplement the problem with the following initial conditions:
\[
  \bu(\cdot,0) = \bu_0, \quad {\rm in } \; \Omega_F, \qquad p(\cdot,0) = p_0, \quad {\rm in } \; \Omega_D.
\]
Here $u_0$ and $p_0$ are the initial velocity and pressure, respectively.

\subsubsection{Weak formulation}
To accommodate for boundary conditions, we introduce the following function spaces:
\begin{align*}
&\mathcal{V}=\{ \bu \in H^1(\Omega_F)^d\; : \; \bu = \boldsymbol0 \; \textrm{on} \; \Gamma_F^1 \},
&\mathcal{X}=\{ \psi \in H^1(\Omega_D) \; : \; \psi =0 \; \textrm{on} \;  \Gamma_D^1 \}.
\end{align*}
The weak solution  to the sharp interface Stokes-Darcy problem is defined as follows. 

\begin{definition}[Weak solution of the sharp interface problem]\label{WS}
The triple $(\bu,p,\pi)$ such that
\begin{align*}
  \bu &\in L^\infty(0,T; L^2(\Omega_F)) \cap L^2(0,T;\mathcal{V}), &\partial_t\bu &\in  L^2(0,T; \mathcal{V}^*),  \\
  p &\in L^\infty(0,T; L^2(\Omega_D)) \cap L^2(0,T;\mathcal{X}), &\partial_t p & \in L^2(0,T; \mathcal{X}^*), \\
  \pi &\in H^{-1}(0,T;L^2(\Omega_F)),
\end{align*}
is a weak solution to our problem if
\[
  \bu(\cdot,0) = \bu_0, \quad {\rm in } \; L^2(\Omega_F), \qquad p(\cdot,0) = p_0, \quad {\rm in } \; L^2(\Omega_D),
\]
and, in addition, for every $(\bv,\psi,\varphi)\in \mathcal{V} \times \mathcal{X} \times L^2(\Omega_F)$,  the following equality is satisfied in $\mathcal{D}'(0,T)$:
\begin{equation}
\label{weakSI}
  \begin{aligned}
    \frac{d}{dt}\left(
        \rho \langle \bu , \bv \rangle_{\mathcal{V}} + c_0 \langle p, \psi \rangle_{\mathcal{X}}
      \right )
    +2\mu\int_{\Omega_F}\bD(\bu):\bD(\bv)+\int_{\Omega_D} \boldsymbol{\kappa}\nabla p \cdot \nabla \psi
    -\int_{\Omega_F}\pi\nabla\cdot\bv +\int_{\Omega_F}\varphi\nabla\cdot\bu
    \\ 
    -\int_{\Gamma} \psi\bu\cdot\bn +\int_{\Gamma} p \bv \cdot \bn 
    +\alpha_{BJ} \sum_{i=1}^{d-1} \int_{\Gamma}(\bu \cdot\bftau_i)(\bv\cdot\bftau_i)
    =\rho\int_{\Omega_F}\bF\cdot\bv+\int_{\Omega_D}g\psi.
  \end{aligned}
\end{equation}
\end{definition}

We comment that, owing to the Gelfand triple framework we are adopting, the initial conditions are meaningful as stated; see \cite[Lemma 7.3]{Roubicek2013}. To simplify notation, we introduce
\[
  \mathcal{H} = \left\{ \bv \in L^2(\Omega_F)^d \; : \; \nabla \cdot \bv = 0, \; {\rm in } \; \Omega_F, \; \bv \cdot \bn = 0 \; {\rm on} \; \Gamma_F^1 \right\}.
\]
We recall that the theory of weak solutions for the  Stokes-Darcy system is by now well understood, see, e.g., \cite{discacciati2002mathematical} for the stationary case and \cite{cao2010coupled} for the evolutionary case.

\begin{theorem}[Well-posedness]\label{thm:WellPosedSharpInterface}
For every $(\bu_0,p_0,\bF,g) \in \mathcal{H} \times L^2(\Omega_D) \times L^2(\Omega_F)^d \times L^2(\Omega_D)$ there is a unique solution to our problem in the sense of Definition~\ref{WS}. Moreover, this solution satisfies the energy equality
\begin{multline*}
  \frac{1}{2}\frac{d}{dt}\left(
    \rho\|\bu\|^2_{L^2(\Omega_F)^d} +c_0\|p\|^2_{L^2(\Omega_D)}
  \right)
  +2\mu\|\bD(\bu)\|^2_{L^2(\Omega_F)^{{d \times d}}}
  +\alpha_{BJ} \sum_{i=1}^{d-1} \| \bu \cdot\bftau_i \|^2_{L^2(\Gamma)}
  +\|\boldsymbol \kappa^{\frac12} \nabla p\|^2_{L^2(\Omega_D)^{{d}}}= \\
  \rho\int_{\Omega_F} \bF \cdot \bu  + \int_{\Omega_D} g p.
\end{multline*}
\end{theorem}
\begin{proof}
It follows after minor modifications to the arguments of \cite{cao2010coupled}.
\end{proof}

\subsection{The diffuse interface formulation}\label{sub:PF}

\begin{figure}
  \begin{center}
    \begin{tikzpicture}
      \draw[red,very thick] (-5,0) -- (5,0)
        node at (5.25,0) {$\Omega$};
      \draw[red,very thick] (0,-2.5) -- (0,2.5)
        node at (0.2,2.5) {$\Gamma$};
      \draw[red,thick,dotted] (0.5,-1) -- (0.5,2.5)
        node at (0.7,2.5) {$\eps$};
      \draw[red,thick,dotted] (-0.5,-1) -- (-0.5,2.5)
        node at (-0.75,2.5) {$-\eps$};
      \draw[red,thick,dotted] (-0.5,1) --(0.5,1)
        node at (-0.7,1) {$\frac1{2}$};
      \draw[red,thick,dotted] (-0.5,2) --(0.5,2)
        node at (-0.7,2) {$1$};

      \draw[black,very thick] (-5,0) -- (-0.5,0);
      \draw[black,very thick] (0.5,2) -- (5,2)
        node at (5.2,2.2) {$\Phi^\eps$};
      \draw[black,very thick, smooth, domain=-0.5:0] plot ({\x, sqrt(2.*\x + 1 ) });
      \draw[black,very thick, smooth, domain=0:0.5] plot ({\x, 2-sqrt(1-2.*\x ) });
      
      \draw[black,very thick] (-5,2) -- (-0.5,2)
        node at (-5.2,2.2) {$\Psi^\eps$};;
      \draw[black,very thick] (0.5,0) -- (5,0);
      \draw[black,very thick, smooth, domain=0:0.5] plot ({\x, sqrt(-2.*\x + 1 ) });
      \draw[black,very thick, smooth, domain=-0.5:0] plot ({\x, 2-sqrt(1+2.*\x ) });

      \draw[blue,very thick,|-] (0,-0.5) -- (5, -0.5) 
        node at (3,-0.3) {$\Omega_F$};
      \draw[blue,very thick,|-] (-0.5,-1) -- (5, -1) 
        node at (3,-0.8) {$\Omega_F^\eps$};
        
      \draw[blue,very thick, |-|] (0, -1.5) -- (0.5,-1.5)
        node at (0.75,-1.5) {$\ell_F^\eps$};
      \draw[violet,very thick, |-|] (-0.5, -2) -- (0,-2)
        node at (-0.75,-2) {$\ell_D^\eps$};
      
      \draw[violet, very thick, -|] (-5,0.2) -- (0,0.2)
        node at (-3,0.4) {$\Omega_D$};
      \draw[violet, very thick, -|] (-5,0.7) -- (0.5,0.7)
        node at (-3,0.9) {$\Omega_D^\eps$};
    \end{tikzpicture}    
  \end{center}
\caption{Graphical representation of the diffuse interface approach. $\Omega_F$ and $\Omega_F^\eps$ are, respectively, the sharp and diffuse fluid domains. Similarly, $\Omega_D$ and $\Omega_D^\eps$ are the sharp and diffuse Darcy domains. $\Phi^\eps$ and $\Psi^\eps$ are the distance functions that define the diffuse fluid and Darcy domains, respectively. Finally, $\ell_F^\eps$ and $\ell_D^\eps$ are transitional layers.}
\label{fig:DomainsAndLayers}
\end{figure}

We now describe our diffuse interface approach. We begin by defining the function
\begin{equation}
  \cS(t) = \begin{dcases}
           -1, & t \leq -1, \\
           (t+1)^\alpha -1, & t \in (-1,0], \\
           1-(1-t)^\alpha, & t \in (0, 1], \\
           1, & t>1.
         \end{dcases}
\label{eq:SfunctionBurgers}
\end{equation}
The diffuse domains will be defined via (sub)level sets of rescaled shifts of this function. Namely, let $\eps>0$ and we define, with $\alpha \in (0,1)$,
\begin{equation}
  \Phi^\eps(x) = \frac12\left[ 1 + \mathcal{S}\left( \frac{\sdist_\Gamma(x)}\eps\right) \right], \qquad\qquad 
  \Psi^\eps(x) = 1-\Phi^\eps(x),
\label{eq:defofPhiPsi}
\end{equation}
where $\sdist_\Gamma$ denotes the signed distance function to $\Gamma$, which is positive on $\Omega_F$. 
We then have that, for $\eps$ sufficiently small, $\Phi^{\eps}\approx 1$ in $\Omega_F$, $\Phi^{\eps}\approx 0$ in $\Omega_D$, and $\Phi^{\epsilon}$ transitions between these two values on a ``diffuse'' layer of width $\mathcal{O}(\epsilon)$. A similar reasoning applies to $\Psi^\eps$. To quantify the fact that we now have a diffuse interface, we introduce the following domains, see Figure~\ref{fig:DomainsAndLayers},
\begin{equation}
\label{eq:DefOfDiffuseDomains}
  \begin{aligned}
	\Omega^{\eps}_F &=\{ x \in \Omega: \Phi^{\epsilon}(x)>0\},
	&\Omega^{\epsilon}_D &=\{x \in \Omega: \Psi^{\epsilon}(x)>0\}, \\
    \ell_F^\eps &= \left\{x \in \Omega_F^\eps : \Phi^\eps(x) \in (\tfrac1{2},1) \right\},
    &\ell_D^\eps &= \left\{x \in \Omega_D^\eps : \Psi^\eps(x) \in (\tfrac1{2},1) \right\},
    &\ell^\eps &= \ell_F^\eps \cup \Gamma \cup \ell_D^\eps.
  \end{aligned}
\end{equation}
Indeed, since $\Omega_F \subset \Omega_F^\eps$ and $\Omega_D \subset \Omega_D^\eps$, these are diffuse versions of our fluid and porous medium domains, respectively. The domains $\ell_F^\eps$ and $\ell_D^\eps$ are transitional layers. Observe that, for $\eps$ sufficiently small, we have
\begin{equation}
\label{eq:MeasureOfLayer}
  |\ell^\eps| = |\ell_F^\eps| + |\ell_D^\eps| \lesssim \eps \mathcal{H}_{d-1}(\Gamma),
\end{equation}
where the implied constant is independent of $\eps$ and $\mathcal{H}_{d-1}(\Gamma)$.

The reason for this particular construction of a phase field function is the following result.

\begin{proposition}[$\Phi^\eps \in A_2$]
The function $\Phi^\eps$, when restricted to $\Omega_F^\eps$, is such that $\Phi^\eps \in A_2$. Similarly, the restriction of $\Psi^\eps$ to $\Omega_D^\eps$ defines an $A_2$ weight.
\end{proposition}
\begin{proof}
We begin by providing an explanation. The way we defined the $A_2$ class, our functions must be defined in the whole $\mathbb{R}^d$, and almost everywhere positive. What we mean here is that there is a, trivial, extension of the given restrictions and that this extension is in the class $A_2$.

Let us now focus on $\Phi^\eps$. The arguments for $\Psi^\eps$ are similar. Notice that this restriction gives $\Phi^\eps \equiv 1$ in $\Omega_F^\eps \setminus \overline{\ell^\eps}$, $\tfrac1{2} \leq \Phi^\eps \leq 1$ in $\ell^\eps_F$, whereas in $\ell^\eps_D$ it is pointwise equivalent to
\[
  \frac1{2} \dist_{\Gamma^\eps}(x)^\alpha , \qquad \Gamma^\eps = \partial \Omega_F^\eps \cap \Omega.
\]
Since $\alpha \in (0,1)$ we see that the power of the distance belongs to $A_2$.
\end{proof}

With the functions $\Phi^\eps$ and $ \Psi^\eps$ at hand, we can present the following heuristics behind our diffuse domain formulation. Notice, first of all, that $\Phi^\eps \equiv 1$ in $\Omega_F\setminus \ell_F^\eps$, and $\Phi^\eps \equiv 0$ in $\Omega_D \setminus \ell_D^\eps$, with a similar behavior for $\Psi^\eps$. Therefore, for $F:\Omega_F \to \reals$ and $G : \Omega_D \to \reals$, we have
\[
  \int_{\Omega_F} F = \int_\Omega F \chi_{\Omega_F} \approx \int_{\Omega_F^\eps} F^\eps \Phi^\eps, \qquad
  \int_{\Omega_D} G = \int_\Omega G \chi_{\Omega_D} \approx \int_{\Omega_D^\eps} G^\eps \Psi^\eps, 
\]
where $F^\eps$ and $G^\eps$ are suitable extensions of $F$ and $G$, respectively, and $\chi_{\Omega_F}, \chi_{\Omega_D}$ are the characteristic functions of $\Omega_F$ and $\Omega_D$. We also need a way to describe quantities that are only defined on $\Gamma$ in a diffuse context. We recall that  $\bn = - \nabla \dist_\Gamma$. Therefore, if $\eps$ is sufficiently small, we have for $\bw : \Gamma \to \reals^d$,
\[
  \int_\Gamma \bw\cdot \bn \approx - \frac1{2\eps} \int_{\ell^\eps } \bw^\eps \cdot \nabla \dist_\Gamma,
\]
where, again, $\bw^\eps$ is a suitable extension of $\bw$. Heuristically $\ell^\eps$ is a layer of thickness $\eps$ on each side of $\Gamma$, hence the factor $\tfrac1{2\eps}$ makes this approximation dimensionally correct. Finally, we must approximate integrals of scalars defined on $\Gamma$. Since $d\Gamma \approx \tfrac1{2\eps}|\nabla \dist_\Gamma| dx$; for $W : \Gamma \to \reals$ we have
\[
  \int_\Gamma W \approx \frac1{2\eps} \int_{\ell^\eps} W^\eps |\nabla \dist_\Gamma|.
\]
As before, $W^\eps$ is a suitable extension. Notice that, in particular, this allows us to write
\[
  \alpha_{BJ} \sum_{i=1}^{d-1} \int_{\Gamma}(\bu \cdot\bftau_i)(\bv\cdot\bftau_i)   \approx
  \frac{ \alpha_{BJ} }{2\eps} \sum_{i=1}^{d-1} \int_{\ell^\eps}(\bu \cdot \tilde{\bftau}_i)(\bv\cdot \tilde{\bftau}_i) |\nabla\dist_\Gamma|,
\]
where the tangent vectors $\tilde{\bftau}_i$ are obtained directly from the phase-field function using the algorithm described in~\cite{martin2000practical,stoter2017diffuse}.

We would like to note that in all approximations listed above, the integrals on the right-hand sides can be equivalently written over the entire domain $\Omega$. In that case, the formulation is more suitable for numerical implementation, as described in Section~\ref{Sec:numerics}. 

\begin{remark}[diffuse interface integrals]
In several references, see for instance \cite{burger2017analysis}, a diffuse normal to the interface is defined by $\bn^\eps = \tfrac{\nabla \Phi^\eps}{|\nabla \Phi^\eps|}$. This, in principle, would only entail some modifications to the way we define the diffuse interface integrals that we have described above. Notice, however, that $\mathcal{S}'$ is not defined when $|t|=1$, and so $\nabla \Phi^\eps$ is not defined either. For this reason, and to more closely follow our numerical implementation, we have kept the distance function, and its gradient, in the diffuse interface integrals.
\end{remark}

%
%

To formulate and analyze our diffuse domain problem, we define, for $k\in\N$, following spaces:
\begin{align*}
  \mathcal{V}^{k,\eps} &= \left\{ \bu \in H^k(\Omega^{\eps}_F, \Phi^{\eps})^d \; : \; \bu = \boldsymbol0 \; \textrm{on} \;  \Gamma_F^1 \right\},  
  &\mathcal{V}^{\eps} &= \mathcal{V}^{1,\eps},
  \\
  \mathcal{X}^{k,\eps} &= \left\{ \psi \in H^k(\Omega_D^{\eps}, \Psi^{\eps}) \; : \; \psi =0 \; \textrm{on} \;  \Gamma_D^1 \right\},
  & \mathcal{X}^{\eps}&=\mathcal{X}^{1,\eps}, \\
  \mathcal{M}^{k,\eps} &= H^k(\Omega_F^\eps, \Phi^\eps), &\mathcal{M}^\eps &= \mathcal{M}^{0,\eps}, \\
  \mathcal{Q}^{k,\eps} &= H^k(\Omega_F^\eps, 1/\Phi^\eps), &\mathcal{Q}^\eps &= \mathcal{Q}^{0,\eps}.
\end{align*}
We comment that, owing to the trace properties of weighted Sobolev spaces we mentioned in Section~\ref{sec:Notation}, these definitions are meaningful and, in addition, all these spaces are Hilbert and separable.
The spaces $\mathcal{V}^{\eps}$, $\mathcal{X}^{\eps}$, and $\mathcal{M}^\eps$ will be associated with the fluid velocity, the porous pressure, and the fluid pressure, respectively. The space $\mathcal{Q}^\eps$ is auxiliary. To handle integrals at the diffuse interface we shall need the following variant of \cite[Theorem 4.2]{burger2017analysis}.

\begin{lemma}[Trace inequality]\label{nablaPhi}
Let $\epsilon_0$ be sufficiently small. Then, there exists a constant $C_{tr}>0$ such that for $0< \epsilon < \epsilon_0$ and for $\bv \in \mathcal{V}^{\eps}$, we have:
\begin{align*}
    \frac1{2\eps} \int_{\ell^\eps} |\bv|^2 |\nabla \dist_\Gamma | \leq C_{tr} \| \bv \|^2_{\mathcal{V}^\eps}.
\end{align*}
A similar estimate hols for $\psi \in \mathcal{X}^\eps$.
\end{lemma}
\begin{proof}
We begin by applying \cite[Theorem 4.2]{burger2017analysis} with the function $S$ of \cite[Example 3.1(i)]{burger2017analysis}, i.e., $S(t) =t$ for $|t|<1$. In the notation of that result this means that, on $\ell^\eps$,
\[
  \varphi^\eps = \frac1\eps \dist_\Gamma, \qquad \omega^\eps = \frac12 \left( 1 + \varphi^\eps \right),
\]
and 
\[
  \frac1{2\eps} \int_{\ell^\eps} |\bv|^2 |\nabla \dist_\Gamma | = \int_{\Omega} |\bv|^2 |\nabla \omega^\eps |\leq C \int_{\Omega_F^\eps} \left[ |\bv|^2 + |\nabla \bv|^2 \right] \omega^\eps = C\int_{\Omega_F^\eps} \left[ |\bv|^2 + |\nabla \bv|^2 \right] \frac{ 1+ \min\{1, \dist_\Gamma/\eps\}}{2}.
\]
Notice now that, $\dist_\Gamma/\eps \leq 1$ if and only if $\tfrac{ \dist_\Gamma + \eps}{2\eps}\leq 1$. Moreover, since $\alpha \in (0,1)$, we have $t \leq t^\alpha$ whenever $t \in (0,1)$, i.e., $\omega^\eps \leq \Phi^\eps$. In summary,
\[
  \frac1{2\eps} \int_{\ell^\eps} |\bv|^2 |\nabla \dist_\Gamma | \leq C_{tr}\int_{\Omega_F^\eps}\left[ |\bv|^2 + |\nabla \bv|^2 \right] \Phi^\eps,
\]
as claimed.
\end{proof}

We are now ready to introduce the notion of weak solution for the diffuse interface problem. We shall assume that there is $\eps_0>0$ such that, for every $\eps \in (0,\eps_0]$, we have $\bu_0^\eps$, $\bF^\eps$, $\boldsymbol{\kappa}_\eps$, $p_0^\eps$, and $g^\eps$, which are suitable extensions of $u_0$, $\bF$, $\boldsymbol\kappa$, $p_0$, and $g$ to $\Omega_F^\eps$ and $\Omega_D^\eps$, respectively.

As a last preliminary step, we introduce some notation. The bilinear form $\cA_\eps: [\mathcal{V}^{\epsilon}\times\mathcal{X}^{\epsilon}]^2 \to \R$ is
\begin{align*}
  \cA_\eps((\bu,p),(\bv,\psi)) &=
  2\mu\int_{\Omega_F^\eps}\bD(\bu):\bD(\bv)\Phi^\eps
  +\int_{\Omega_D^\eps} \boldsymbol\kappa_\eps\nabla p \cdot \nabla \psi  \Psi^\eps 
  \\
  &+\frac1{2\eps}\int_{\ell^\eps} \psi \bu \cdot \nabla \dist_\Gamma
  -\frac1{2\eps} \int_{\ell^\eps} p \bv \cdot \nabla \dist_\Gamma
  +\frac{ \alpha_{BJ} }{2\eps} \sum_{i=1}^{d-1} \int_{\ell^\eps}(\bu \cdot \tilde{\bftau}_i)(\bv\cdot \tilde{\bftau}_i) |\nabla\dist_\Gamma|.
\end{align*}
Notice that, with the aid of Lemma~\ref{nablaPhi}, all the integrals in this definition make sense. The bilinear forms $\cB_\eps : \mathcal{V}^\eps \times \mathcal{M}^\eps \to \R$ and $\cC_\eps : \mathcal{V}^\eps \times \mathcal{Q}^\eps \to \R$ are, respectively,
\[
  \cB_\eps(\bv,\pi) = -\int_{\Omega_F^\eps} \nabla \cdot \bv \pi \Phi^\eps, \qquad
  \cC_\eps(\bv,\pi) = -\int_{\Omega_F^\eps} \nabla \cdot \bv \pi.
\]
Finally, the linear form $\mathcal{F}_\eps: \mathcal{V}^{\epsilon}\times\mathcal{X}^{\epsilon} \to \R$ is defined as:
\[
  \langle\mathcal{F}_\eps,(\bv,\psi)\rangle = \rho\int_{\Omega_F^\eps}\bF^\eps\cdot\bv\Phi^\eps+\int_{\Omega_D^\eps}g^\eps\psi  \Psi^\eps .
\]

Let us now define our notion of solution.

\begin{definition}[Weak solution of the diffuse interface problem]\label{SDFP_def}
We say that the triple $(\bu^{\eps},p^{\eps},\pi^{\eps})$ with
\begin{align*}
  \bu^{\eps} &\in L^\infty(0,T;L^2(\Omega_F^\eps, \Phi^\eps)) \cap L^2(0,T;\mathcal{V}^{\eps}), &\partial_t \bu^\eps &\in L^2(0,T;(\mathcal{V}^\eps)^* ) ,\\
  p^{\eps} &\in L^\infty(0,T;L^2(\Omega_D^\eps, \Psi^\eps )) \cap L^2(0,T;\mathcal{X}^{\eps}), &\partial_t p^\eps &\in L^2(0,T;(\mathcal{X}^\eps)^* ), \\
  \pi^{\eps} &\in H^{-1}(0,T;\mathcal{M}^\eps),
\end{align*} 
is a diffuse interface weak solution to our problem if
\[
  \bu^\eps(\cdot,0)= \bu_0^\eps, \quad {\rm in} \; L^2(\Omega_F^\eps,\Phi^\eps), \qquad p^\eps(\cdot,0) = p_0^\eps, \quad {\rm in} \; L^2(\Omega_D^\eps, \Psi^\eps ),
\]
and, in addition, for every $(\bv,\psi,\varphi)\in \mathcal{V}^{\eps} \times \mathcal{X}^{\eps} \times L^2(\Omega^{\eps}_F)$,  the following equality is satisfied in $\mathcal{D}'(0,T)$:
\begin{equation}
\label{SDFP_weak}
  \frac{d}{dt}\left(
   \rho \langle \bu^{\eps} , \bv \rangle_{\mathcal{V}^\eps} +c_0\langle p^{\eps}, \psi \rangle_{\mathcal{X}^{\eps}}  
  \right)
  + \cA_\eps( (\bu^\eps,p^\eps), (\bv,\psi) )
  + \cB_\eps(\bv,\pi^\eps) - \cB_\eps( \bu^\eps, \varphi )
  = \langle\mathcal{F}_\eps,(\bv,\psi)\rangle .
\end{equation}
\end{definition}

We observe that the initial conditions are meaningful owing, once again, to the Gelfand triple structure we have adopted.

Our main goal in this paper is to study the convergence of suitable discretizations of \eqref{SDFP_weak} to solutions of our problem in the sharp interface sense, i.e., according to Definition~\ref{WS}. This will be done in two steps. In the first step, in Section \ref{Sec_error}, we fix $\epsilon$ and the phase field functions $\Phi^\eps, \Psi^\eps$, and prove error estimates for a finite element approximation of the diffuse interface formulation~\eqref{SDFP_weak}. In the second step, in Section \ref{Sec_ModellinError}, we analyze the convergence of the continuous diffuse interface  formulation to the continuous sharp interface formulation.

\subsection{Well-posedness}\label{sub:DiffuseDomainWellPosed}

Let us prove the well-posedness for the diffuse interface formulation. Owing to all the existing theory regarding weighted Sobolev spaces, the proof is similar to the well-posedness proof for the Stokes-Darcy system (e.g. \cite{cao2010coupled,DisQuart09} and Theorem~\ref{thm:WellPosedSharpInterface}). Therefore, here we just outline the main steps of the proof. We define the function spaces
\begin{align*}
  \mathcal{H}^\eps_{div} &= \left\{ \bu \in L^2(\Omega_F^\eps, \Phi^\eps)^d\; : \; \nabla\cdot\bu = 0 \; \textrm{in} \; \Omega_F^\eps \;, \; \bu\cdot \bn = 0 \; \textrm{on} \;  \Gamma_F^1 \right\},
\qquad
  \mathcal{V}^{\epsilon}_{div} &=\left\{\bv \in \mathcal{V}^{\eps} : \nabla\cdot\bv=0 \right\}.
\end{align*}
Notice that, by definition, $\bu^\eps \in L^2(0,T;\mathcal{V}^{\epsilon}_{div})$. First we prove the following lemma.

\begin{lemma}[Coercivity]
The  bilinear  form  $\cA_\eps(\cdot, \cdot) $ is  continuous  and  coercive  on $\mathcal{V}^{\epsilon}\times\mathcal{X}^{\epsilon}$.
\label{bilina}
\end{lemma}
\begin{proof}
In order to show continuity, special attention has to be given to the terms arising from the coupling at the diffuse interface. The Cauchy-Schwartz inequality followed by Lemma~\ref{nablaPhi} yield:
\begin{align*}
  \frac1{2\eps}\int_{\ell^\eps} p \bv \cdot \nabla\dist_\Gamma
  & \leq
  \left[  \frac1{2\eps} \int_{\ell^\eps} |p|^2 |\nabla \dist_\Gamma| \right]^{\frac12}
  \left[  \frac1{2\eps}\int_{\ell^\eps} |\bv|^2 |\nabla \dist_\Gamma| \right]^{\frac12}  
  \leq
  C_{tr}  \|p\|_{\mathcal{X}^\eps}
  \|\bv\|_{\mathcal{V}^\eps}.
\end{align*}
Similarly,  we have
\begin{align*}
  \frac1{2\eps} \int_{\ell^\eps} \psi \bu \cdot \nabla\dist_\Gamma
  \leq
  C_{tr}  \|\psi\|_{\mathcal{X}^\eps}
  \|\bu\|_{\mathcal{V}^\eps},
\end{align*}
and
\begin{align*}
  \frac1{2\eps} \sum_{i=1}^{d-1} \int_{\ell^\eps}(\bu \cdot \tilde{\bftau}_i)(\bv\cdot \tilde{\bftau}_i) |\nabla\dist_\Gamma |
  &    \leq 
  \sum_{i=1}^{d-1}
  \left(  \frac1{2\eps} \int_{\ell^\eps} \left|\bu \cdot \tilde{\bftau}_i \right|^2 |\nabla \dist_\Gamma | \right)^{\frac12}
  \left(  \frac1{2\eps} \int_{\ell^\eps} \left|\bv \cdot \tilde{\bftau}_i \right|^2 |\nabla \dist_\Gamma | \right)^{\frac12} 
  \\
  &    \leq
  \sum_{i=1}^{d-1} C_{tr}
  \left\|\bu \cdot \tilde{\bftau}_i \right\|_{H^1(\Omega_F^\eps, \Phi^\eps)}
  \left\|\bv \cdot \tilde{\bftau}_i \right\|_{H^1(\Omega_F^\eps, \Phi^\eps)}
  \leq C_{tr}
  \left\|\bu \right\|_{\mathcal{V}^\eps}
  \left\|\bv \right\|_{\mathcal{V}^\eps}.
\end{align*}
Other terms can be bounded in a standard way. Therefore, we obtain the following estimate:
\begin{align*}
  \cA_\eps((\bu,p),(\bv,\psi))\leq  &  
    2\mu \|\bD(\bu)\|_{L^2(\Omega_F^\eps, \Phi^\eps)} \|\bD(\bv)\|_{L^2(\Omega_F^\eps, \Phi^\eps)} 
  +\|\boldsymbol\kappa_\eps^{1/2} \nabla p \|_{L^2(\Omega_D^\eps,  \Psi^\eps )} \| \boldsymbol\kappa_\eps^{1/2}\nabla \psi \|_{L^2(\Omega_D^\eps,  \Psi^\eps )} 
  \\
  &  
  + C_{tr} \|p\|_{\mathcal{X}^\eps} \|\bv\|_{\mathcal{V}^\eps}
  + C_{tr} \|\psi\|_{\mathcal{X}^\eps} \|\bu\|_{\mathcal{V}^\eps}
  +\alpha_{BJ} C_{tr} \left\|\bu \right\|_{\mathcal{V}^\eps} \left\|\bv \right\|_{\mathcal{V}^\eps}
  \\
  \leq  & 
  D_1  \left\|(\bu,p) \right\|_{\mathcal{V}^{\epsilon}\times\mathcal{X}^{\epsilon}}
  \left\|(\bv,\psi) \right\|_{\mathcal{V}^{\epsilon} \times\mathcal{X}^{\epsilon}},
\end{align*}
where
\[
  D_1=\max\{ \alpha_{BJ} C_{tr}, 2 \mu, k^{*}, C_{tr} \}.
\]
To show coercivity, we set $(\bv,\psi) =(\bu,p) $ to obtain:
\begin{align*}
  \cA_\eps((\bu,p), (\bu,p)) = & 2\mu\int_{\Omega_F^\eps}|\bD(\bu)|^2 \Phi^\eps
  +\int_{\Omega_D^\eps} \boldsymbol\kappa_\eps |\nabla p |^2  \Psi^\eps 
  +\frac{\alpha_{BJ}}{2\eps} \sum_{i=1}^{d-1} \int_{\ell^\eps}(\bu \cdot \tilde{\bftau}_i)^2  |\nabla\dist_\Gamma|
  \\
  \geq &
  2\mu \|\bD(\bu)\|^2_{L^2(\Omega_F^\eps, \Phi^\eps)}
  + \|\boldsymbol\kappa_\eps^{\frac12} \nabla p \|^2_{L^2(\Omega_D^\eps,  \Psi^\eps )}.
\end{align*}
Using the weighted Korn inequality given in \eqref{eq:weightedKorn} and the weighted Poincar\'e-Friedrichs-type inequality of \eqref{eq:WeightedPoincare}, we have
\[
  \cA_\eps((\bu,p), (\bu,p)) 
  \geq
  D_2  \left\|(\bu,p)\right\|^2_{\mathcal{V}^{\epsilon}\times\mathcal{X}^{\epsilon}},
  \qquad
  D_2=\min\left\{ \frac{2\mu}{C_K}, \frac{k_*}{C_P}\right\},
\]
which completes the proof.
\end{proof}

We can now prove well-posedness.

\begin{theorem}[Well-posedness]\label{WellPosednes}
There is $\eps_0>0$ such that, for every $\eps \in (0, \eps_0]$, problem \eqref{SDFP_weak} is well-posed in the following sense. Let $\boldsymbol\kappa_\eps \in L^\infty(\Omega_D^\eps)$ be uniformly bounded and positive definite. Then, for every $\bu_0^\eps \in \mathcal{H}^\eps_{div}$, $p_0^\eps \in L^2(\Omega_D^\eps, \Psi^\eps )$, $\bF^\eps \in L^2(\Omega_F^\eps,\frac1{\Phi^\eps})$ and $g^\eps \in L^2(\Omega_D^\eps,\frac1{ \Psi^\eps })$, problem \eqref{SDFP_weak} has a unique solution. Moreover, this solution satisfies the energy identity
\begin{multline*}
  \frac{1}{2}\frac{d}{dt}\left(
    \rho\|\bu^\eps\|^2_{L^2(\Omega_F^\eps,\Phi^\eps)^d} +c_0\|p^\eps\|^2_{L^2(\Omega_D^\eps,   \Psi^\eps )}
  \right)
  +2\mu\|\bD(\bu^\eps)\|^2_{L^2(\Omega_F^\eps,\Phi^\eps)^{{d \times d}}}
  +\frac{\alpha_{BJ}}{2\eps} \sum_{i=1}^{d-1} \int_{\ell^\eps} | \bu^\eps \cdot\tilde \bftau_i |^2 |\nabla \dist_\Gamma| \\
  +\|\boldsymbol \kappa_\eps^{\frac12} \nabla p^\eps\|^2_{L^2(\Omega_D^\eps, \Psi^\eps )^{{d}}}= 
  \rho\int_{\Omega_F^\eps} \bF^\eps \cdot \bu^\eps \Phi^\eps+ \int_{\Omega_D^\eps} g^\eps p^\eps  \Psi^\eps .
\end{multline*}
\end{theorem}
\begin{proof}
From  estimates analogous to the ones used in Lemma~\ref{bilina}, it follows that $\mathcal{F}_\eps \in \left(\mathcal{V}_{div}^{\epsilon}\times\mathcal{X}^{\epsilon}\right)^*$. Therefore, by using standard methods (e.g., Galerkin method), one can easily prove that there exists a unique solution $(\bu^\eps,p^\eps)$ to the following evolution problem in $\mathcal{D}'(0,T)$:
\begin{align}\label{VProblem}
  \frac{d}{dt}\left(
    \rho  \langle \bu^\eps , \bv \rangle_{\mathcal{V}^\eps} +c_0 \langle p, \psi\rangle_{\mathcal{X}^\eps}
  \right)
  +\cA_\eps((\bu^\eps,p^\eps),(\bv,\psi))=\langle\mathcal{F}_\eps,(\bv,\psi)\rangle,
  \quad (\bv,\psi)\in \mathcal{V}^{\epsilon}_{div}\times\mathcal{X}^{\epsilon}.
\end{align}
Moreover, this solution satisfies the claimed energy estimate.

To finish the proof it remains to construct the pressure $\pi^\eps \in H^{-1}(0,T;\mathcal{M}^\eps)$ corresponding to the solution $(\bu^\eps,p^\eps)$ of \eqref{VProblem}. To begin, assume that there is a constant $B>0$ such that, for all $\psi \in \mathcal{Q}^\eps$,
\begin{equation}
\label{eq:inf-sup-no-Phi}
  B\| \psi \|_{\mathcal{Q}^\eps} \leq \sup_{\boldsymbol0 \neq \bv \in \mathcal{V}^\eps} \frac{\cC_\eps(\bv,\psi)}{ \|\bv \|_{\mathcal{V}^\eps}}.   
\end{equation}
In this case, an argument similar to that of the Navier-Stokes equations (see, e.g., \cite[Proposition III.1.1.]{TemamBook}) shows that there is a unique $r^\eps \in H^{-1}(0,T;\mathcal{Q}^\eps)$ such that \eqref{VProblem} can be equivalently rewritten as
\[
  \frac{d}{dt}\left(
    \rho  \langle \bu^\eps , \bv \rangle_{\mathcal{V}^\eps} +c_0 \langle p, \psi\rangle_{\mathcal{X}^\eps}
  \right)
  +\cA_\eps((\bu^\eps,p^\eps),(\bv,\psi))- \langle\mathcal{F}_\eps,(\bv,\psi)\rangle = 
  -\cC_\eps(\bv,r^\eps),
  \quad (\bv,\psi)\in \mathcal{V}^{\epsilon}\times\mathcal{X}^{\epsilon},  
\]
where the test function $\bv$ is not assumed solenoidal anymore. Notice now that, for any domain $D$ and any weight $\omega \in A_2$, the mapping
\[
  L^2(D,\omega^{-1}) \ni r \mapsto r \omega^{-1} \in L^2(D,\omega)
\]
is an isometric isomorphism between these spaces. We then define $\pi^\eps = r^\eps/\Phi^\eps \in H^{-1}(0,T;\mathcal{M}^\eps)$ to conclude that this is the, necessarily unique, function that satisfies
\[
  \frac{d}{dt}\left(
    \rho  \langle \bu^\eps , \bv \rangle_{\mathcal{V}^\eps} +c_0 \langle p, \psi\rangle_{\mathcal{X}^\eps}
  \right)
  +\cA_\eps((\bu^\eps,p^\eps),(\bv,\psi))- \langle\mathcal{F}_\eps,(\bv,\psi)\rangle = 
  -\cB_\eps(\bv,\pi^\eps),
  \quad (\bv,\psi)\in \mathcal{V}^{\epsilon}\times\mathcal{X}^{\epsilon}.  
\]
Adding to this the easy identity
\[
  -\cB_\eps(\bu^\eps,\varphi)
  = 0,
\]
which is valid for almost every time, and every $\varphi \in \mathcal{M}^\eps$, we obtain that the triple $(\bu^\eps,p^\eps,\pi^\eps)$ is the weak solution in the sense of Definition~\ref{SDFP_def}.

It remains then to prove \eqref{eq:inf-sup-no-Phi}. Note that we may not simply invoke \eqref{eq:Bogovskii} as the function $\psi$ does not have zero average, and the functions $\bv$ do not have the correct boundary values. Instead, we employ an argument inspired by \cite[Lemma 3.1]{NochettoVaca}. We proceed then in several steps.
\begin{enumerate}[1.]
  \item Let $\psi \in \mathcal{Q}^\eps$. Decompose it as
  \[
    \psi = \psi_1 + \psi_2, \qquad \psi_2 = \int_{\Omega_F^\eps} \psi.
  \]
  Notice that
  \[
    \int_{\Omega_F^\eps} \psi_1 = 0, \quad \int_{\Omega_F^\eps} \psi_1 \psi_2 = 0, 
    \quad \| \psi_i \|_{\mathcal{Q}^\eps} \leq C_i \| \psi \|_{\mathcal{Q}^\eps},
  \]
  where the constants $C_i$ depend only on $[\Phi^\eps]_{A_2}$ and $|\Omega_F^\eps|$.
  
  \item Notice now that, owing to \eqref{eq:Bogovskii}, there is $\bv_1 \in H^1_0(\Omega_F^\eps,\Phi^\eps)^d$ such that
  \[
    \int_{\Omega_F^\eps} \nabla \cdot \bv_1 \psi_1 = \| \psi_1 \|_{\mathcal{Q}^\eps}^2, \qquad
    \| \nabla \bv_1 \|_{L^2(\Omega_F^\eps, \Phi^\eps)^{d \times d}} \leq \frac1\beta \| \psi_1 \|_{\mathcal{Q}^\eps}.
  \]
  
  \item \label{it:StepW} Denote by $\Gamma_\eps = \partial\Omega_F^\eps \setminus \Gamma_F^1$. Recall that $\psi_2$ is a constant. Now, choose $\bw \in \mathcal{V}^\eps$ such that, if $\bn_{\Gamma_\eps}$ denotes the unit normal to $\Gamma_\eps$,
  \[
    \int_{\Omega_F^\eps} \nabla \cdot \bw = \int_{\Gamma_\eps} \bw \cdot \bn_{\Gamma_\eps} = a_\bw \neq 0, \qquad 
    \| \nabla \bw \|_{L^2(\Omega_F^\eps,\Phi^\eps)^{d \times d}} = b_\bw \neq 0.
  \]
  This is possible because, owing to \eqref{eq:TraceIsL2}, we have $\bw\cdot\bn_{\Gamma_\eps} \in L^2(\Gamma_\eps)$.
  
  Define
  \[
    \bv_2 = \frac1{a_\bw} \left( \int_{\Omega_F^\eps} \frac1{\Phi^\eps} \right)^{1/2} \signum \psi_2 \| \psi_2 \|_{\mathcal{Q}^\eps}\bw
  \]
  to get that
  \[
    \| \nabla \bv_2 \|_{L^2(\Omega_F^\eps,\Phi^\eps)^{d \times d}} = \frac{b_\bw}{a_\bw}\left( \int_{\Omega_F^\eps} \frac1{\Phi^\eps} \right)^{1/2}  \| \psi_2 \|_{\mathcal{Q}^\eps} = \frac1\gamma \| \psi_2 \|_{\mathcal{Q}^\eps},
  \]
 where $\gamma=\displaystyle\frac{a_\bw}{b_\bw}  \left( \int_{\Omega_F^\eps} \frac1{\Phi^\eps} \right)^{-1/2}$. More importantly, since $\psi_2$ is a constant,
  \[
    \int_{\Omega_F^{\eps}} \psi_2 \nabla \cdot \bv_2 = \psi_2 \int_{\Gamma_\eps} \bv_2\cdot\bn = \| \psi_2 \|_{\mathcal{Q}^\eps}^2.
  \]
  
  \item Let $\bv=\bv_1 + \gamma^2 \bv_2 \in \mathcal{V}^\eps$. Notice then than
  \[
    \| \nabla \bv \|_{L^2(\Omega_F^\eps,\Phi^\eps)^{d \times d}} \leq \left( \frac{C_1}\beta + C_2\gamma \right)\| \psi \|_{\mathcal{Q}^\eps}.
  \]
  In addition,
  \begin{align*}
    \int_{\Omega_F^\eps} \psi \nabla \cdot \bv &= \int_{\Omega_F^\eps} \psi_1 \nabla \cdot \bv_1 + \gamma^2 \int_{\Omega_F^\eps} \psi_1 \nabla \cdot \bv_2 + \psi_2 \int_{\Omega_F^\eps}  \nabla \cdot \bv_1 +\gamma^2 \int_{\Omega_F^\eps} \psi_2 \nabla \cdot \bv_2 \\
    &= \| \psi_1 \|_{\mathcal{Q}^\eps}^2 + \gamma^2 \| \psi_2 \|_{\mathcal{Q}^\eps}^2 + \gamma^2 \int_{\Omega_F^\eps} \psi_1 \nabla \cdot \bv_2 + \psi_2 \int_{\Gamma_\eps}  \bv_1 \cdot \bn \\
    &= \| \psi_1 \|_{\mathcal{Q}^\eps}^2 + \gamma^2 \| \psi_2 \|_{\mathcal{Q}^\eps}^2 + \gamma^2 \int_{\Omega_F^\eps} \psi_1 \nabla \cdot \bv_2,
  \end{align*}
  where we used that $\bv_1 \in H^1_0(\Omega_F^\eps, \Phi^\eps )^d$. Now, a combination of Young's inequality and the norm estimate on $\bv_2$ gives
  \[
    \gamma^2 \int_{\Omega_F^\eps} \psi_1 \nabla \cdot \bv_2 \geq 
    - \frac12 \| \psi_1 \|_{\mathcal{Q}^\eps}^2 - \frac{\gamma^2}2 \| \psi_2 \|_{\mathcal{Q}^\eps}^2,
  \]
  so that
  \[
    \int_{\Omega_F^\eps} \psi \nabla \cdot \bv \geq \frac12 \min\{1,\gamma^2\} \left( \| \psi_1 \|_{\mathcal{Q}^\eps}^2 + \| \psi_2 \|_{\mathcal{Q}^\eps}^2 \right) \geq \frac14 \min\{1,\gamma^2\} \| \psi \|_{\mathcal{Q}^\eps}^2.
  \]

  \item The previous estimates show that \eqref{eq:inf-sup-no-Phi} holds with
  \[
    B \geq \frac{\frac14 \min\{1,\gamma^2\}}{ \frac{C_1}\beta + C_2\gamma }. 
  \]
\end{enumerate}
This concludes the proof.
\end{proof}

\section{Discretization}\label{sec:Discretization}

Having obtained the well posedness of our diffuse interface problem, in this section we proceed with its discretization. We shall provide a numerical scheme, and an error analysis for it.

\subsection{Time discretization}\label{sub:TimeDiscr}
For time discretization we will employ the backward Euler method. 
Let $N \in \mathbb{N}$ be the number of time steps, and $\dt = T/N$ the timestep. We define, for $n = 0, \ldots, N$, the discrete times $t^n = n \dt$. Let $X$ be a normed space. For a given function $W :[0,T] \to X$, we will compute sequences $W^{\dt} = \{ W^n\}_{n=0}^N$ so that $W^n \approx W(t^n)$. The discrete time derivative is defined as
\[
  d_t W^{n+1} = \frac{W^{n+1}-W^n}{\Delta t}.                                                                                                                                                                                                                                                                                   
\]
Over such sequences, we define the following norms
\[
  \| W^\dt \|_{L^2_\dt(0,T;X)} = \left( \dt\sum_{ n=0}^{N-1} \| W^{n+1} \|_X^2 \right)^{1/2},
  \qquad
  \| W^\dt \|_{L^\infty_\dt(0,T;X)} = \max_{0\leq n \leq N} \| W^n \|_X.
\]

\subsection{Space discretization}\label{sub:FEM}
To discretize in space, we use the finite element method. Since we have assumed that $\Om$ is a polytope, it can be meshed exactly. Let $\{\Th\}_{h>0}$ be a family of conforming, simplicial triangulations of $\Om$ that is quasiuniform in the usual finite element sense \cite{ciarlet1978finite,MR851383}. The parameter $h>0$ denotes the mesh size. We assume that, for each $h>0$, we have at hand spaces
\[
  \mathcal{V}_h^\eps \subset W^{1,\infty}(\Omega_F^\eps) \subset \mathcal{V}^\eps, \qquad
  \mathcal{X}_h^\eps \subset W^{1,\infty}(\Omega_D^\eps) \subset \mathcal{X}^\eps, \qquad
  \mathcal{Q}_h^\eps \subset L^\infty(\Omega_F^\eps) \subset \mathcal{Q}^\eps,
\]
which consist of piecewise polynomials subordinate to the mesh $\Th$. These spaces will be used to approximate the fluid velocity, Darcy pressure, and a quantity related to the fluid pressure, respectively. These spaces are assumed to possess suitable approximation properties. This is quantified by the existence of operators
\[
  I_{\mathcal{V}} : \mathcal{V}^\eps \to \mathcal{V}^\eps_h, \qquad
  I_{\mathcal{X}} : \mathcal{X}^\eps \to \mathcal{X}^\eps_h, \qquad
  I_{\mathcal{Q}} : \mathcal{Q}^\eps \to \mathcal{Q}_h^\eps,
\]
that are stable, i.e., there is $C>0$ independent of $h$ such that, for all $(\bv,q,\theta) \in \mathcal{V}^\eps \times \mathcal{X}^\eps \times \mathcal{Q}^\eps$,
\begin{equation}
  \| I_{\mathcal{V}} \bv \|_{\mathcal{V}^\eps} + \| I_{\mathcal{X}} q \|_{\mathcal{X}^\eps} + \| I_{\mathcal{Q}} \theta \|_{\mathcal{Q}^\eps}
    \leq 
  C \left( \| \bv \|_{\mathcal{V}^\eps} + \|  q \|_{\mathcal{X}^\eps} + \| \theta \|_{\mathcal{Q}^\eps} \right);
\label{eq:interpolApprox}
\end{equation}
and, in addition, yield optimal approximation. This is expressed by the existence of $k \in \N$ such that, if $(\bv,q,\theta) \in \mathcal{V}^{k+1,\eps} \times \mathcal{X}^{k+1,\eps} \times \mathcal{Q}^{k,\eps}$
\[
  \| \bv-I_{\mathcal{V}} \bv \|_{\mathcal{V}^\eps} + \| q-I_{\mathcal{X}} q \|_{\mathcal{X}^\eps} + \| \theta-I_{\mathcal{Q}} \theta \|_{\mathcal{Q}^\eps}
    \leq 
  C h^k \left( \| \bv \|_{\mathcal{V}^{k+1,\eps}} + \|  q \|_{\mathcal{X}^{k+1,\eps}} + \| \theta \|_{\mathcal{Q}^{k,\eps}} \right).
\]
We comment that, owing to the fact that $\Phi^\eps, \Psi^\eps \in A_2$, reference \cite{nochetto2016piecewise} has shown that standard, piecewise polynomial, finite element spaces possess all the aforementioned properties. This reference has also provided an explicit construction of the requisite operators.

We, finally, must require that a discrete version of \eqref{eq:inf-sup-no-Phi} holds uniformly in $h$. Namely, there is $b>0$ such that, for all $h>0$,
\begin{equation}
\label{eq:DiscrInfSupWeirdBCs}
    b\| q_h \|_{\mathcal{Q}^\eps} \leq \sup_{\boldsymbol0 \neq \bv_h \in \mathcal{V}^\eps_h} \frac{ \cC_\eps(\bv_h,q_h)}{ \| \bv_h \|_{\mathcal{V}^\eps}}, \qquad \forall q_h \in \mathcal{Q}_h^\eps.   
\end{equation}
To achieve this, we recall that \cite{duran2020stability} has shown that, essentially, any classically inf--sup stable pair $(\mathcal{V}_h^\eps,\mathcal{Q}_h^\eps)$ satisfies a discrete version of \eqref{eq:Bogovskii}. This, in particular means that there is a mapping $\tilde{P}_h : H^1_0(\Omega_F^\eps,\Phi^\eps)^d \to \mathcal{V}_h^\eps \cap H^1_0(\Omega_F^\eps,\Phi^\eps)^d$ such that, for all $\bv \in H^1_0(\Omega_F^\eps,\Phi^\eps)^d$
\begin{equation}
\label{eq:discrBogovskii}
  \int_{\Omega_F^\eps} q_h \nabla\cdot (\bv - \tilde{P}_h \bv ) = 0, \quad \forall q_h \in \mathcal{Q}_h^\eps, \ \int_{\Omega_F^\eps} q_h =0, \qquad \| \grad \tilde{P}_h \bv \|_{L^2(\Omega_F^\eps,\Phi^\eps)^{d \times d}} \leq \frac1{\tilde{\beta}}\| \grad \bv \|_{L^2(\Omega_F^\eps,\Phi^\eps)^{d \times d}}.
\end{equation}
Let us show that, under this assumption, our finite element spaces are also compatible in the sense that is needed for our purposes.

\begin{lemma}[Compatibility]
Assume that the pair of finite element spaces $(\mathcal{V}_h^\eps,\mathcal{Q}_h^\eps)$ satisfies \eqref{eq:discrBogovskii}. If $h$ is sufficiently small, then \eqref{eq:DiscrInfSupWeirdBCs} holds. This, in particular, implies that there is an operator
$
  P_h : \mathcal{V}^\eps \to \mathcal{V}_h^\eps
$
which is a projection, i.e., $P_h^2 = P_h$, it is stable,  has optimal approximation properties, meaning,
\[
  \| P_h \bv \|_{\mathcal{V}^\eps} \leq \frac1b\| \bv \|_{\mathcal{V}^\eps},
  \qquad
  \| \bv - P_h \bv \|_{\mathcal{V}^\eps} \leq \left( 1 + \frac1b \right)\| \bv - I_{\mathcal{V}} \bv \|_{\mathcal{V}^\eps},
\]
and it satisfies
\[
  \int_{\Omega_F^\eps} q_h \nabla\cdot (\bv - {P}_h \bv ) = 0, \quad \forall q_h \in \mathcal{Q}_h^\eps.
\]
\end{lemma}
\begin{proof}
To begin, we remark that the approximation properties of $P_h$, if it exists, immediately follow from its stability and the fact that it is a projection. Next, we recall that the existence of $P_h$ as a consequence of \eqref{eq:DiscrInfSupWeirdBCs} is known in the literature as the Fortin criterion \cite{MR851383}.

It remains then to prove \eqref{eq:DiscrInfSupWeirdBCs}. To achieve this, we will follow the same ideas that allowed us to obtain \eqref{eq:inf-sup-no-Phi}. Indeed, we can employ \eqref{eq:discrBogovskii} instead of \eqref{eq:Bogovskii}. The only subtle point is that, in Step~\ref{it:StepW}, one may be tempted to choose an arbitrary $\bw_h \in \mathcal{V}_h^\eps$. However, this may lead the involved constants to be dependent of $h$. We then, instead, need to set $\bw_h = I_{\mathcal{V}} \bw$, where $\bw \in \mathcal{V}^\eps$ is the function in that proof. Then, for small enough $h$, we must have
\[
  \frac12 b_\bw \leq \| \bw \|_{\mathcal{V}^\eps} - \| \bw - \bw_h \|_{\mathcal{V}^\eps} \leq \| \bw_h \|_{\mathcal{V}^\eps} \leq C\| \bw \|_{\mathcal{V}^\eps} = Cb_\bw
\]
and 
\[
  \left| \int_{\Gamma_\eps} \bw_h \cdot \bn - a_\bw \right| \leq |\Gamma_\eps|^{1/2} \| \bw - \bw_h \|_{L^2(\Gamma_\eps)} \leq 
  C_{tr}|\Gamma_\eps|^{1/2} \| \bw - \bw_h \|_{\mathcal{V}^\eps} \leq \frac14 a_\bw,
\]
where we used the trace inequality of \cite{MR1258430}.

This concludes the proof.
\end{proof}

\subsection{The scheme}\label{sub:TheNumScheme}
We are now ready to present the scheme and discuss its basic properties. We begin by setting
\[
  (\bu_h^0 , p_h^0 ) = (P_h\bu_0^\eps , I_{\mathcal{X}} p_0^\eps )\in \mathcal{V}^\eps_h \times \mathcal{X}^\eps_h.
\]
Then, for $n = 0, \ldots, N-1$, we seek for $(\bu_h^{n+1}, p_h^{n+1}, \theta_h^{n+1} )\in \mathcal{V}^\eps_h \times \mathcal{X}^\eps_h \times \mathcal{Q}_h^\eps$ such that, for every $(\bv_{h},\psi_h,\varphi_h) \in \mathcal{V}_{h}^{\eps} \times \mathcal{X}^{\eps}_{F,h} \times \mathcal{Q}_h^\eps$, we have:
\begin{multline}
\label{approxWF}
  \rho \int_{\Omega_F^\eps} d_t \bu_{h}^{n+1} \cdot\bv_{h} \Phi^\eps +c_0\int_{\Omega_D^\eps} d_t p_{h}^{n+1} \psi_{h}  \Psi^\eps 
  + \cA_\eps( (\bu_h^{n+1},p_h^{n+1}), (\bv_h,\psi_h) ) \\
  + \cC_\eps(\bv_h,\theta_h^{n+1}) -\cC_\eps(\bu_h^{n+1},\varphi_h)  =
  \langle \mathcal{F}_\eps, (\bv_h,\psi_h) \rangle
\end{multline}
For details on implementation, the reader is referred to see Section~\ref{Sec:numerics}.

\begin{remark}[fluid pressure]
Notice that, in scheme \eqref{approxWF}, the terms involving incompressibility and the fluid pressure are missing the weight $\Phi^\eps$. This is due to the fact that we do have the compatibility condition \eqref{eq:DiscrInfSupWeirdBCs}, but not a similar one involving the weight. Nevertheless, if one wishes to compute an approximate pressure, this can be obtained a posteriori via
\[
  \pi_h^\dt = \frac1{\Phi^\eps} \theta_h^\dt.
\]
\end{remark}

Owing to the structure of the scheme we immediately have the following existence, uniqueness, and stability result for \eqref{approxWF}.

\begin{theorem}[Discrete stability]\label{thm:DiscrStability}
In the setting of Theorem~\ref{WellPosednes} we have that, for every $\dt$ and $h$, scheme \eqref{approxWF} has a unique solution. Moreover, this solution satisfies the following energy identity 
\begin{multline*}
  \frac{\dt}{2}\left[
    \rho d_t \|\bu_h^{n+1}\|^2_{L^2(\Omega_F^\eps,\Phi^\eps)^d} + \rho \dt \| d_t \bu_h^{n+1} \|^2_{L^2(\Omega_F^\eps,\Phi^\eps)^d}
    +c_0 d_t \|p_h^{n+1}\|^2_{L^2(\Omega_D^\eps, \Psi^\eps )} + c_0 \dt \|d_t p_h^{n+1}\|^2_{L^2(\Omega_D^\eps, \Psi^\eps )}
  \right] \\
  +2\mu \dt\|\bD(\bu_h^{n+1})\|^2_{L^2(\Omega_F^\eps,\Phi^\eps)^{{d \times d}}}
  +\frac{\alpha_{BJ}}{2\eps} \dt \sum_{i=1}^{d-1} \int_{\ell^\eps} | \bu_h^{n+1} \cdot\tilde \bftau_i |^2 |\nabla \dist_\Gamma|
  +\dt \|\boldsymbol \kappa_\eps^{\frac12} \nabla p^\eps\|^2_{L^2(\Omega_D^\eps, \Psi^\eps )^{{d}}}= \\
  \rho\dt\int_{\Omega_F^\eps} \bF^\eps \cdot \bu_h^{n+1} \Phi^\eps + \dt \int_{\Omega_D^\eps} g^\eps p_h^{n+1}  \Psi^\eps .
\end{multline*}
\end{theorem}
\begin{proof}
The proof mimics, with discrete arguments, that of Theorem~\ref{WellPosednes}. Indeed, coercivity of the bilinear form $\cA_\eps$ is transferred to $\mathcal{V}_h^\eps \times \mathcal{X}_h^\eps$. In addition, the compatibility condition \eqref{eq:DiscrInfSupWeirdBCs} yields the existence and uniqueness of $\theta_h^{k+1}$.

Finally, to obtain the energy identity, it suffices to set $(\bv_h,\psi_h,\varphi_h) = \dt (\bu_h^{n+1},p_h^{n+1},\theta_h^{n+1})$ and apply the classical polarization identity.
\end{proof}

\subsection{Error analysis}\label{Sec_error}

We now proceed to carry out an error analysis for scheme \eqref{approxWF}. We will do so under the assumption that, for some $k \in \mathbb{N}$, we have
\begin{equation}
\label{eq:SolIsSmooth}
  \bu^\eps \in C^2([0,T]; \mathcal{V}^{k+1,\eps}), \qquad p^\eps \in C^2([0,T]; \mathcal{X}^{k+1,\eps}), \qquad \pi^\eps \in C^1([0,T]; \mathcal{M}^{k,\eps}).
\end{equation}
These smoothness assumptions are just made for simplicity. Similar results can be obtained under less stringent conditions. We begin by constructing sequences $\bu^\dt = \{ \bu^{n} = \bu^\eps(t^n) \}_{n=0}^N$, $p^\dt = \{ p^{n} = p^\eps(t^n) \}_{n=0}^N$, and $\pi^\dt = \{ \pi^{n} = \pi^\eps(t^n) \}_{n=0}^N$.
Notice that, owing to the assumed smoothness, we have
\begin{equation*}
  \dt \sum_{n=1}^{N-1} \| d_t \bu^{n+1} \|^2_{L^2(\Omega_F^\eps,\Phi^\eps)^d} \lesssim \| \partial_t \bu^\eps \|^2_{L^2(0,T; L^2(\Omega_F^\eps,\Phi^\eps)^d)},
\end{equation*}
with a similar estimate for $p^\dt$. The main result of this section is an error analysis for \eqref{approxWF}.

We begin by observing that, upon defining $\theta = \pi^\eps \Phi^\eps \in L^2(0,T;\mathcal{Q}^\eps)$, problem \eqref{SDFP_weak} can be rewritten as
\begin{equation}
    \frac{d}{dt}\left(
   \rho \langle \bu^{\eps} , \bv \rangle_{\mathcal{V}^\eps} +c_0\langle p^{\eps}, \psi \rangle_{\mathcal{X}^{\eps}}  
  \right)
  + \cA_\eps( (\bu^\eps,\psi^\eps), (\bv,\psi) )
  + \cC_\eps(\bv,\theta) - \cC_\eps( \bu^\eps, \varphi )
  = \langle\mathcal{F}_\eps,(\bv,\psi)\rangle ,
\label{eq:NewWeakDiffuse}
\end{equation}
where now $\varphi \in \mathcal{Q}^\eps$. It is this form of our problem that can be compared to the scheme \eqref{approxWF}. We begin with a time-consistency estimate.

\begin{lemma}[Consistency]\label{lem:ConsistencyDerivatives}
For $n = 0, \ldots, N-1$ define $\cR_\eps^{n+1} \in ( \mathcal{V}^\eps \times \mathcal{X}^\eps )^*$ via
\[
  \langle \cR_\eps^{n+1}, (\bv,\psi) \rangle = \rho \int_{\Omega_F^\eps} \left( d_t\bu - \partial_t \bu \right)^{n+1} \cdot \bv \Phi^\eps + c_0 \int_{\Omega_D^\eps} \left( d_t p - \partial_t p \right)^{n+1} \psi  \Psi^\eps .
\]
Then, assuming \eqref{eq:SolIsSmooth}, we have
\[
  \| \cR_\eps^\dt \|_{L^2_\dt(0,T;( \mathcal{V}^\eps \times \mathcal{X}^\eps )^*)} \lesssim \dt,
\]
where the hidden constants depend on $(\bu,p)$ and $T$, but not on the discretization parameters.
\end{lemma}
\begin{proof}
Using \eqref{eq:SolIsSmooth} the proof is standard and follows from a Taylor expansion.
\end{proof}

We can now proceed with our main error estimate.

\begin{theorem}[Error estimate]\label{MainThm}
Let $(\bu^\eps,p^\eps,\pi^\eps)$ solve \eqref{SDFP_weak} and satisfy \eqref{eq:SolIsSmooth}. Let $(\bu_h^\dt,p_h^\dt,\theta_h^\dt)$ solve \eqref{approxWF} with initial data
\[
  (\bu^{0}_h, p^0_h ) = (P_h \bu_{0}^\eps, I_{\mathcal X} p_0^{\eps}).
\]
The following error estimate holds
\begin{equation}
  \begin{aligned}
    \| \bu^\dt - \bu_h^\dt \|_{L^\infty_\dt(0,T;L^2(\Omega_F^\eps,\Phi^\eps)^d)} + \| p^\dt - p_h^\dt \|_{L^\infty_\dt(0,T;L^2(\Omega_D^\eps,  \Psi^\eps ))} &\lesssim \dt + h^{k+1}, \\
    \| \bu^\dt - \bu_h^\dt \|_{L^2_\dt(0,T;\mathcal{V}^\eps)} + \| p^\dt - p_h^\dt \|_{L^2_\dt(0,T;\mathcal{X}^\eps)} &\lesssim \dt + h^{k},
  \end{aligned}
\end{equation}
where the hidden constants depend on the material parameters, higher order smoothness norms on the solution, $T$, $\Omega$, $[\Phi^\eps]_{A_2}$, and $[\Psi^\eps]_{A_2}$. However, they are independent of $\dt$ and $h$.
\end{theorem}
\begin{proof}
Owing to the linearity of the problem, the assumed smoothness of the data, the compatibility condition \eqref{eq:DiscrInfSupWeirdBCs}, and the approximation properties of finite elements on weighted spaces, the proof is rather standard and proceeds, for instance, like the proof of error estimates for the time dependent Stokes problem; see, for instance \cite[Chapter 6]{MR2050138}.

Define $\be_\bu^\dt = \bu^\dt - \bu_h^\dt$, $e_p^\dt  = p^\dt - p_h^\dt$, and $e_\theta^\dt = \theta^\dt - \theta_h^\dt$. Upon restricting, in \eqref{eq:NewWeakDiffuse}, the test functions to lie on the corresponding finite element spaces, we get
\begin{multline*}
  \rho \int_{\Omega_F^\eps} d_t \be_{\bu}^{n+1} \cdot\bv_{h} \Phi^\eps +c_0\int_{\Omega_D^\eps} d_t e_{p}^{n+1} \psi_{h}  \Psi^\eps 
  + \cA_\eps( (\be_{\bu}^{n+1},e_p^{n+1}), (\bv_h,\psi_h) ) \\
  + \cC_\eps(\bv_h,e_\theta^{n+1}) -\cC_\eps(\be_\bu^{n+1},\varphi_h)  =
  \langle \mathcal{R}_\eps^{n+1}, (\bv_h,\psi_h) \rangle,
\end{multline*}
where $\cR^\dt_\eps$ was defined in Lemma~\ref{lem:ConsistencyDerivatives}.

Let now $(\bU_h^\dt,\Theta_h^\dt) \subset \mathcal{V}_h^\eps \times \mathcal{Q}_h^\eps$ denote the $\Phi^\eps$--weighted Stokes projection of $(\bu^\dt,\theta^\dt)$. Similarly, $P_h^\dt \subset \mathcal{X}_h^\eps$ is the $ \Psi^\eps $--weighted Ritz projection of $p^\dt$. For these quantities we have, for all $n = 0, \ldots, N-1$,
\[
  \cA_\eps( (\bU_h^{n+1} - \bu^{n+1},P_h^{n+1} - p^{n+1}), (\bv_h,\psi_h) ) 
  + \cC_\eps(\bv_h,\Theta_h^{n+1} - \theta^{n+1}) -\cC_\eps(\bU_h^{n+1} - \bu^{n+1},\varphi_h)  = 0,
\]
where $(\bv_h,\psi_h) \in \mathcal{V}_h^\eps \times \mathcal{X}_h^\eps$ but are, otherwise, arbitrary. We now, as it is by now standard, further decompose the error as
\begin{align*}
  \be_{\bu}^\dt &= \bu^\dt - \bU_h^\dt + \bE_{\bu,h}^\dt, &\bE_{\bu,h}^\dt &= \bU_h^\dt - \bu_h^\dt, \\
  e_p^\dt  &= p^\dt - P_h^\dt + E_{p,h}^\dt, &E_{p,h}^\dt &= P_h^\dt - p_h^\dt, \\
  e_\theta^\dt &= \theta^\dt - \Theta_h^\dt + E_{\theta,h}^\dt, &E_{\theta,h}^\dt &= \Theta_h^\dt - \theta_h^\dt.
\end{align*}
This decomposition then yields
\begin{multline*}
  \rho \int_{\Omega_F^\eps} d_t \bE_{\bu,h}^{n+1} \cdot\bv_{h} \Phi^\eps +c_0\int_{\Omega_D^\eps} d_t E_{p,h}^{n+1} \psi_{h}  \Psi^\eps 
  + \cA_\eps( (\bE_{\bu,h}^{n+1},E_{p,h}^{n+1}), (\bv_h,\psi_h) ) \\
  + \cC_\eps(\bv_h,E_{\theta,h}^{n+1}) -\cC_\eps(\bE_{\bu,h}^{n+1},\varphi_h)  =
  \langle \mathcal{R}_\eps^{n+1}, (\bv_h,\psi_h) \rangle -  \\
  \rho \int_{\Omega_F^\eps} d_t \left( \bu - \bU_h\right)^{n+1} \cdot\bv_{h} \Phi^\eps -c_0\int_{\Omega_D^\eps} d_t \left( p - P_h \right)^{n+1} \psi_{h}  \Psi^\eps .
\end{multline*}

It is now time to recall that, owing to \cite{duran2020stability}, the Stokes projection $(\bU_h^\dt,\Theta_h^\dt)$ is stable, and thus possesses optimal approximation properties on weighted spaces. Similarly, \cite[Section 4]{MR4265062}, has shown that the Ritz projection $P_h^\dt$ is stable and possesses optimal approximation properties on weighted spaces.

The rest of the proof proceeds then in a standard way. We invoke the discrete stability of Theorem~\ref{thm:DiscrStability}, the interpolation estimates of \eqref{eq:interpolApprox}, the consistency error of Lemma~\ref{lem:ConsistencyDerivatives}, and the triangle inequality.
\end{proof}

\section{Modeling error}\label{Sec_ModellinError}
In this section, we estimate the difference between the continuous solutions of the diffuse domain  and sharp interface formulations in terms of the parameter $\eps$ describing the diffuse interface width. 

We will prove two types of convergence results. First, we will assume that the weights are strictly positive and Lipschitz. Note that this assumption is not satisfied by the weights used in Section~\ref{sub:PF}. However, in numerical computation we need to introduce a regularization $\delta>0$ which makes the weights strictly positive on the whole domain and Lipschitz. Such regularization does not affect the results from the previous sections. Indeed, with strictly positive Lipschitz weights the proofs of the results in Section~\ref{sub:PF} are classical and do not use weighted spaces. However, to see that all estimates are uniform in terms of the regularization parameter $\delta>0$, one must carry out an analysis without the regularization parameter as it was done in Section~\ref{sub:PF}. Numerical experiments confirm our analysis and show that our error estimates do not depend on $\delta$.
Another reason why we start the modeling error analysis with regularized weights is to present the main ideas in a less technical way. Here we will rely on results from \cite[Section 5]{burger2017analysis} to estimate difference between sharp and diffuse integrals.

The second set of results concerns weights without regularization, as introduced in Section~\ref{sub:PF}. We note that the results from \cite[Section 5]{burger2017analysis} are true also for this case, but with appropriately adjusted convergence rate. Therefore, we can prove this stronger result in an analogous way with slight technical complications.

\subsection{Positive Lipschitz weights}
Let us now construct the regularized phase field functions. Let $\cS_r$ be a function that satisfies conditions (S1)--(S3) from \cite[Section 3]{burger2017analysis}. An example of such function is given by
\[
  \cS_r(t) = \begin{dcases}
               t, & |t| \leq 1, \\
               \frac{t}{|t|}, & |t| > 1.
             \end{dcases}
\]
The phase field functions $\Phi^\eps$ and $\Psi^\eps$ are given as in \eqref{eq:defofPhiPsi}, but with $\cS$ replaced by $\cS_r$. Let $\delta>0$ be a small regularization parameter. We define a regularized phase field function in the following way:
\begin{align}
  \Phi^{\eps,\delta} = (1-2 \delta) \Phi^{\eps} + \delta =\Phi^{\eps}+\delta(1-2\Phi^{\eps}),
  \qquad
  \Psi^{\eps,\delta} = 1- \Phi^{\eps,\delta}.
\label{regularizacija}
\end{align}
Notice that since both $\Phi^{\eps,\delta}$ and $\Psi^{\eps,\delta}$ are strictly positive they satisfy all conditions from Sections \ref{sub:DiffuseDomainWellPosed} and \ref{sec:Discretization} with $\Omega_F^{\eps}=\Omega_D^{\eps}=\Omega$.
To make the exposition more clear, we slightly change the notation and introduce subscripts ``\emph{s}'' and ``\emph{d}'' to denote the solutions to the sharp interface and the diffuse interface formulations, respectively. 

The weak solution, $(\bu_s,p_s),$ to the sharp interface problem satisfies the following weak formulation (see Definition \ref{WS}):
\begin{equation}
\label{weakSI2}
  \begin{aligned}
	\rho \int_{\Omega_F}\partial_t \bu_s \cdot\bv
	+c_0\int_{\Omega_D}\partial_t p_s \psi  
	+\int_{\Omega_F}\bfsigma(\nabla\bu_s,\pi_s):\nabla \bv
	+\int_{\Omega_D} \boldsymbol \kappa\nabla p_s \cdot \nabla \psi 
	\\
	+\int_{\Gamma} p_s \bv \cdot \bn
	-\int_{\Gamma} \psi \bu_s \cdot \bn
	+\alpha_{BJ} \sum_{i=0}^{d-1}\int_{\Gamma}(\bu_s \cdot {\bftau}_i)(\bv\cdot {\bftau}_i)
	= \rho \int_{\Omega_F}\bF\cdot\bv+\int_{\Omega_D}g\psi,
	\quad (\bv,\psi)\in \mathcal{V}_{div}\times\mathcal{X}.
  \end{aligned}
\end{equation}
Here the subscript $div$ in $\mathcal{V}_{div}$ denotes the space of divergence free functions. The continuous diffuse interface solution, with the  diffuse layer of width $\mathcal{O}(\epsilon)$ and regularization $\delta,$ is denoted by $(\bu^{\epsilon,\delta}_d,p^{\epsilon,\delta}_d)$, and satisfies the following weak formulation (Definition \ref{SDFP_def}):
\begin{equation}
  \begin{aligned}
	\rho\int_{\Omega}\partial_t \bu^{\epsilon,\delta}_d \cdot\bv\Phi^{\epsilon,\delta}
	+c_0\int_{\Omega} \partial_t p^{\epsilon,\delta}_d \psi \Psi^{\epsilon,\delta}  
	+\int_{\Omega} \bfsigma(\nabla\bu^{\epsilon,\delta}_d,\pi^{\epsilon,\delta}_d):\nabla \bv \Phi^{\epsilon,\delta}
	+\int_{\Omega} \boldsymbol\kappa\nabla p^{\epsilon,\delta}_d \cdot \nabla \psi \Psi^{\epsilon,\delta}
	\\ 
	-\int_{\Omega} p^{\epsilon,\delta}_d \bv \cdot \nabla\Phi^{\epsilon,\delta}
	+\int_{\Omega} \psi \bu^{\epsilon,\delta}_d \cdot \nabla\Phi^{\epsilon,\delta}
	+\alpha_{BJ} \sum_{i=0}^{d-1}\int_{\Omega}(\bu^{\epsilon,\delta}_d\cdot \tilde{\bftau}_i)(\bv\cdot \tilde{\bftau}_i) |\nabla\Phi^{\epsilon,\delta}| 
	= \rho \int_{\Omega}\bF\cdot\bv\Phi^{\epsilon,\delta}+\int_{\Omega}g\psi \Psi^{\eps,\delta},    
  \end{aligned}
\label{weakDI}
\end{equation}
for all test functions $(\bv,\psi)\in \mathcal{V}^{\eps,\delta}_{div}\times\mathcal{X}^{\eps,\delta}$, where the superscripts on the definition of the corresponding function spaces carry the obvious meaning.
Our goal is to estimate $\|(\bu_s,p_s)-(\bu^{\epsilon,\delta}_d,p^{\epsilon,\delta}_d)\|$ in a suitable norm, in terms of powers of $\epsilon$ and $\delta$. This will be done by subtracting the weak form of the diffuse interface formulation from the weak form of the sharp interface formulation, taking suitable test functions, and using the convergence properties of the diffuse integrals.

\begin{theorem}[Modeling error I]\label{ModellinError}
Let $(\bu_s, p_s, \pi_s)$ be the solution of the sharp interface Stokes-Darcy problem in the sense of Definition \ref{WS}, and $(\bu^{\epsilon,\delta}_d,  p^{\epsilon,\delta}_d, \pi^{\eps,\delta}_d)$ be the solution of the diffuse interface problem in sense of Definition \ref{SDFP_def}. Furthermore,  assume $\epsilon>0$, $\delta\geq 0$, $\bF\in H^1(\Omega)^d$, $g\in H^1(\Omega)$ and
\begin{align*}
  (\bu_s,\pi_s) &\in \left(H^1(0,T;W^{1,\infty}(\Omega)^d)\cap L^2(0,T;W^{3,\infty}(\Omega)^d)\right)  \times L^2(0,T;W^{2,\infty}(\Omega)), \\
  p_s &\in H^1(0,T;W^{1,\infty}(\Omega))\cap L^2(0,T;W^{3,\infty}(\Omega)).
\end{align*}
Then, there exists a constant $C$ independent of $\eps$ and $\delta$ such that the following estimate holds:
\begin{equation}   
  \begin{aligned}
	\rho\|\bu_s-\bu_d^{\epsilon,\delta}\|_{L^{\infty}(0,T;L^2(\Omega,\Phi^{\epsilon,\delta})^d)}^2
	+c_0\|p_s-p^{\epsilon,\delta}_d\|_{L^{\infty}(0,T;L^2(\Omega,\Psi^{\epsilon,\delta}))}^2
	+\nu \|\bD(\bu_s-\bu_d^{\epsilon,\delta})\|^2_{L^2(0,T;L^2(\Omega,\Phi^{\epsilon,\delta})^{d\times d})}
	\\
	+\frac{\boldsymbol \kappa}{2}\|\nabla(p_s-p^{\epsilon,\delta}_d)\|^2_{L^2(0,T;L^2(\Omega,\Psi^{\epsilon,\delta} )^d)}
	+\alpha_{BJ}
	\sum_{i=0}^{d-1}\int_0^T\int_{\Omega}\left |(\bu_s-\bu_d^{\epsilon,\delta})\cdot \tilde{\bftau}_i \right |^2 |\nabla \Phi^{\epsilon,\delta}|
	\\
	\leq C (\epsilon^3+\delta)
	\left (
	\|\partial_t\bu_s\|^2_{L^2(0,T;W^{1,\infty}(\Omega)^d)}
	+\|\partial_t p_s\|^2_{L^2(0,T;W^{1,\infty}(\Omega))}
	\right .
	\\    
	\left .
	+\|\bu_s\|^2_{L^2(0,T;W^{3,\infty}(\Omega))}+\|\pi_s\|^2_{L^2(0,T;W^{2,\infty}(\Omega))}+\|p_s\|^2_{L^2(0,T;W^{3,\infty}(\Omega))}
	\right ).    
  \end{aligned}
\label{ModellinEstimate}
\end{equation}
\end{theorem}
\begin{proof}
To simplify notation we set $\chi = \chi_{\Omega_F}$.
We begin by subtracting \eqref{weakDI} from \eqref{weakSI2}. Using a divergence-free extension operator to extend the sharp interface solution to the whole domain, we obtain the following equation:
\begin{align*}
  \rho  \int_{\Omega}\partial_t (\bu_s-\bu^{\epsilon,\delta}_d)\cdot\bv\Phi^{\epsilon,\delta}
  +  \rho \int_{\Omega}\partial_t\bu_s\cdot \bv (\chi-\Phi^{\epsilon,\delta})
  +c_0\int_{\Omega}\partial_t (p_s-p^{\epsilon,\delta}_d)\psi\Psi^{\epsilon,\delta}
  +c_0\int_{\Omega}\partial_t p_s\psi (\Phi^{\epsilon,\delta}-\chi)
  \\
  +\int_{\Omega}\bfsigma\left(\nabla(\bu_s-\bu_d^{\epsilon,\delta}),\pi_s-\pi_d^{\epsilon,\delta}\right):\nabla \bv\Phi^{\epsilon,\delta}
  +\int_{\Omega}\bfsigma(\nabla\bu_s,\pi_s):\nabla \bv (\chi-\Phi^{\epsilon,\delta})
  +\boldsymbol\kappa\int_{\Omega}\nabla(p_s-p^{\epsilon,\delta}_d)\cdot \nabla\psi \Psi^{\epsilon,\delta}
  \\
  +\boldsymbol\kappa\int_{\Omega}\nabla p_s\cdot\nabla\psi (\Phi^{\epsilon,\delta}-\chi)
  +\int_{\Gamma} p_s \bv \cdot \bn
  -\int_{\Gamma} \psi \bu_s \cdot \bn
  +\alpha_{BJ} \sum_{i=0}^{d-1}\int_{\Gamma}(\bu_s \cdot {\bftau}_i)(\bv\cdot {\bftau}_i)
  \\
  +\int_{\Omega} p^{\epsilon,\delta}_d \bv \cdot \nabla\Phi^{\epsilon,\delta}
  -\int_{\Omega} \psi \bu^{\epsilon,\delta}_d \cdot \nabla\Phi^{\epsilon,\delta}
  - \alpha_{BJ} \sum_{i=0}^{d-1}\int_{\Omega}(\bu^{\epsilon,\delta}_d\cdot \tilde{\bftau}_i)(\bv\cdot \tilde{\bftau}_i) |\nabla\Phi^{\epsilon,\delta}| 
  =\int_{\Omega}\left (\rho\bF\cdot\bv-g\psi\right )(\chi-\Phi^{\eps,\delta}).
\end{align*}
Taking $(\bv,\psi)=(\bu_s-\bu_d^{\epsilon,\delta},p_s-p_d^{\epsilon,\delta})$, we get:
\begin{align}
  \frac{1}{2}\frac{d}{dt}\left (     
  \rho \|\bu_s-\bu^{\epsilon,\delta}_d\|^2_{L^2(\Omega, \Phi^{\eps,\delta})^d}
  + c_0 \|p_s-p^{\epsilon,\delta}_d \|^2_{L^2(\Omega, \Psi^{\eps,\delta})}
  \right)\notag
  \\
  +2\mu \|\bD(\bu_s-\bu_d^{\epsilon,\delta})\|^2_{L^2(\Omega, \Phi^{\eps,\delta})^{d \times d}}
  +\boldsymbol\kappa \|\nabla(p_s-p^{\epsilon,\delta}_d)\|^2_{L^2(\Omega, \Psi^{\eps,\delta})^d}
  \notag
  \\
  =
  \underbrace{-\rho\int_{\Omega}\partial_t\bu_s\cdot \bv (\chi-\Phi^{\epsilon,\delta})
	  -c_0\int_{\Omega}\partial_t p_s \psi (\Phi^{\epsilon,\delta}-\chi)}_{I_1}
  \underbrace{-\int_{\Omega}\bfsigma(\nabla\bu_s,\pi_s):\nabla \bv(\chi-\Phi^{\epsilon,\delta})}_{I_2}
  \notag
  \\
  \underbrace{-\boldsymbol\kappa\int_{\Omega}\nabla p_s\cdot\nabla\psi (\Phi^{\epsilon,\delta}-\chi)}_{I_3}
  \underbrace{-\int_{\Gamma} p_s \bv \cdot \bn
	  +\int_{\Gamma} \psi \bu_s \cdot \bn
	  -\int_{\Omega} p^{\epsilon,\delta}_d \bv\cdot \nabla\Phi^{\epsilon,\delta}
	  +\int_{\Omega} \psi \bu^{\epsilon,\delta}_d \cdot \nabla\Phi^{\epsilon,\delta}}_{I_4}
  \notag
  \\
  \underbrace{ -\alpha_{BJ} \sum_{i=0}^{d-1} \left (\int_{\Gamma}(\bu_s \cdot\bftau_i)(\bv\cdot\bftau_i)-
	  \int_{\Omega}(\bu^{\epsilon,\delta}_d\cdot \tilde{\bftau}_i)(\bv\cdot \tilde{\bftau}_i) |\nabla\Phi^{\epsilon,\delta}| \right )}_{I_5}
  +\underbrace{\int_{\Omega}\left (\rho\bF\cdot\bv-g\psi \right )(\chi-\Phi^{\eps,\delta})}_{I_6}\label{moderror}	
\end{align}
Here we used that $\bu_s$ and $\bu_d^{\epsilon,\delta}$ are divergence free, and therefore:
\[
  \int_{\Omega}\bfsigma(\nabla(\bu_s-\bu_d^{\epsilon,\delta}),\pi_s-\pi_d^{\epsilon,\delta}):\nabla(\bu_s-\bu_d^{\epsilon,\delta})\Phi^{\epsilon,\delta}
  =2\mu\int_{\Omega}|\bD(\bu_s-\bu_d^{\epsilon,\delta})|^2\Phi^{\epsilon,\delta}.
\]

Our goal is to estimate the right-hand side of~\eqref{moderror} in terms of $\epsilon$ and the norms of the sharp interface solution. This will be done by  using the convergence results for the diffuse integrals from \cite[Section 5]{burger2017analysis}, and the following estimates which are consequences of the Poincare's and Korn's inequalities
\[
  \|\bv\|^2_{H^1(\Omega,\Phi^{\epsilon,\delta})^d}\ltt \int_{\Omega}|{\bD}(\bu_s-\bu_d^{\epsilon,\delta})|^2\Phi^{\epsilon,\delta},
  \qquad
  \|\psi\|^2_{H^1(\Omega,\Psi^{\epsilon,\delta})}\ltt \int_{\Omega}|\nabla(p_s-p^{\epsilon,\delta}_d)|^2\Psi^{\epsilon,\delta}.
\]
Notice that because of \eqref{regularizacija} we have $\|\cdot \|_{L^2(\Omega,\Phi^{\eps})}\ltt \|\cdot \|_{L^2(\Omega,\Phi^{\eps,\delta})}$. The first integral is estimated using \eqref{regularizacija} and \cite[Theorem 5.2]{burger2017analysis} for $p=2$:
\begin{align*}
  |I_1| &\leq \left |\rho\int_{\Omega}\partial_t\bu_s\cdot \bv (\chi-\Phi^{\epsilon})
  +c_0\int_{\Omega}\partial_t p_s \psi (\Phi^{\epsilon}-\chi)\right |  
  +\delta\left |\rho\int_{\Omega}\partial_t\bu_s\cdot \bv (-1+2\Phi^{\epsilon})
  +c_0\int_{\Omega}\partial_t p_s \psi (1-2\Phi^{\epsilon})
  \right |
  \\
  &\ltt \epsilon^{3/2}\left (                      
  \rho  \|\partial_t\bu_s\|_{W^{1,\infty}(\Omega)^d}\|\bD(\bu_s-\bu_d^{\epsilon,\delta})\|_{L^2(\Omega,\Phi^{\epsilon})^{d \times d}}
  +  \|\partial_t p_s \|_{W^{1,\infty}(\Omega)}\|\nabla(p_s-p^{\epsilon,\delta}_d)\|_{L^2(\Omega,\Psi^{\epsilon})^d}
  \right )+I_1^{\delta}.
\end{align*}
We estimate $I_1^{\delta}$ as follows:
\begin{align*}
  |I_1^{\delta}| &\ltt \left(\|\partial_t \bu_s \|_{L^{\infty}(\Omega)^d}+\|\partial_t p_s\|_{L^{\infty}(\Omega)}\right )\int_{\Omega}\delta(|\bv|+|\psi|) \\
  &\ltt  \left (\|\partial_t \bu_s\|_{L^{\infty}(\Omega)^d}+\|\partial_t p_s\|_{L^{\infty}(\Omega)}\right )\sqrt{\delta}
  \left ( \|\bv\|_{L^2(\Omega,\Phi^{\eps,\delta})^d}
  + \|\psi\|_{L^2(\Omega,\Psi^{\eps,\delta})} \right  )  
\end{align*}
We note that all other  terms in volume integrals that have a factor of $\delta$ can be estimated analogously. Finally, by using $\|\cdot\|_{L^2(\Omega,\Phi^{\eps})}\ltt \|\cdot\|_{L^2(\Omega,\Phi^{\eps,\delta})}$ and combining previous estimates we get:
\begin{align*}
  |I_1| &\ltt (\eps^{3/2}+\sqrt{\delta})
  \left (\|\partial_t\bu_s\|_{W^{1,\infty}(\Omega)^d}+ \|\partial_t p_s \|_{W^{1,\infty}(\Omega)}\right )
  \left (\|\bD(\bu_s-\bu_d^{\epsilon,\delta})\|_{L^2(\Omega,\Phi^{\epsilon,\delta})^{d \times d}} \right. \\
  &+ \left. \|\nabla(p_s-p^{\epsilon,\delta}_d)\|_{L^2(\Omega,\Psi^{\epsilon,\delta})^d}\right ).
\end{align*}
Similarly,
\begin{align*}
  |I_6|\ltt (\epsilon^{3/2}+\sqrt{\delta})\left (\|\bF\|_{H^1(\Omega)^d}+\|g\|_{H^1(\Omega)}\right )
  \left (\|\bD(\bu_s-\bu_d^{\epsilon,\delta})\|_{L^2(\Omega,\Phi^{\epsilon,\delta})^{d \times d}}
  +\|\nabla(p_s-p^{\epsilon,\delta}_d)\|_{L^2(\Omega,\Psi^{\epsilon,\delta})^d}\right ).
\end{align*}

Notice that by \eqref{regularizacija} $\nabla\Phi^{\eps}$ and $\nabla\Phi^{\eps,\delta}$ are colinear and therefore we can define tangents $\{\tilde{\bftau}_i\}_{i=1}^{d-1}$ using the regularized phase field functions in the same way as before. To estimate $I_2$, we integrate by parts:
\begin{align*}
  I_2&=\int_{\Omega}\nabla\cdot\bfsigma(\nabla\bu_s,\pi_s)\cdot\bv(\chi-\Phi^{\epsilon,\delta})
  -\int_{\Omega}\bfsigma(\nabla\bu_s,\pi_s)\nabla\Phi^{\epsilon,\delta}\cdot\bv
  -\int_{\Gamma}\bfsigma(\nabla\bu_s,\pi_s)\bn\cdot\bv
  \\
  &=\int_{\Omega}\nabla\cdot\bfsigma(\nabla\bu_s,\pi_s)\cdot\bv(\chi-\Phi^{\epsilon,\delta})
  -\int_{\Omega} \left(\bfsigma(\nabla\bu_s,\pi_s)\frac{\nabla\Phi^{\epsilon,\delta}}{|\nabla\Phi^{\epsilon,\delta}|}\cdot\frac{\nabla\Phi^{\epsilon,\delta}}{|\nabla\Phi^{\epsilon,\delta}|} \right)(\bv\cdot\nabla\Phi^{\epsilon,\delta}) 
  \\
  &-\sum_{i=1}^{d-1} \int_{\Omega} \left(\bfsigma(\nabla\bu_s,\pi_s)\tilde{\bftau}_i \cdot\tilde{\bftau}_i \right)(\bv\cdot \tilde{\bftau}_i ) |\nabla\Phi^{\epsilon,\delta}|
  +\int_{\Gamma}p_s\bv\cdot\bn
  +\alpha_{BJ} \sum_{i=1}^{d-1}\int_{\Gamma}(\bu_s\cdot\bftau_i)(\bv\cdot\bftau_i),
\end{align*}
In a similar way, we have:
\begin{align*}
  I_3= \boldsymbol\kappa\int_{\Omega}\Delta p_s\psi(\chi-\Phi^{\epsilon,\delta})
  +\boldsymbol\kappa\int_{\Omega}\psi\nabla p_s\cdot \nabla\Phi^{\epsilon,\delta}
  +\int_{\Gamma} \psi\underbrace{\boldsymbol\kappa \nabla p_s \cdot \bn}_{-\bu_s\cdot\bn},
\end{align*}
\begin{align*}
  I_4&=\underbrace{\int_{\Omega}(p_s-p_d^{\varepsilon})\bv\cdot\nabla\Phi^{\epsilon,\delta}
  -\int_{\Omega}\psi(\bu_s-\bu_d^{\epsilon,\delta})\cdot\nabla\Phi^{\epsilon,\delta}}_{0}
  \\
  &-\int_{\Omega}p_s\bv\cdot\nabla\Phi^{\epsilon,\delta}
  -\int_{\Gamma} p_s \bv\cdot \bn
  +\int_{\Omega}\psi\bu_s\cdot\nabla\Phi^{\epsilon,\delta}
  +\int_{\Gamma} \psi \bu_s \cdot \bn,
\end{align*}
and
\begin{align*}
  I_5=-\alpha_{BJ}  \sum_{i=1}^{d-1} \int_{\Omega}\left |(\bu_s-\bu_d^{\epsilon,\delta})\cdot \tilde{\bftau}_i \right |^2 |\nabla \Phi^{\epsilon,\delta}|
  +\alpha_{BJ}\sum_{i=1}^{d-1} \int_{\Omega}(\bu_s\cdot \tilde{\bftau}_i )(\bv\cdot \tilde{\bftau}_i ) |\nabla\Phi^{\epsilon,\delta}|
  -\alpha_{BJ}  \sum_{i=1}^{d-1} \int_{\Gamma}(\bu_s \cdot\bftau_i)(\bv\cdot\bftau_i).
\end{align*}
Adding the estimates for $I_2, I_3, I_4$ and $I_5$ together, we obtain:
\begin{align*}
\sum_{k=2}^5 I_k &=
-\alpha_{BJ}  \sum_{i=1}^{d-1} \int_{\Omega}\left |(\bu_s-\bu_d^{\epsilon,\delta})\cdot \tilde{\bftau}_i \right |^2 |\nabla \Phi^{\epsilon,\delta}|
+
\int_{\Omega}\left (\nabla\cdot\bfsigma(\nabla\bu_s,\pi_s)\cdot\bv+ \boldsymbol\kappa\Delta p_s\psi\right )(\chi-\Phi^{\epsilon,\delta})
\\
&+\int_{\Omega}\psi\left (\bu_s+\boldsymbol\kappa\nabla p_s   \right )\cdot\nabla\Phi^{\epsilon,\delta}
+\int_{\Omega}\left (-p_s-\bfsigma(\nabla\bu_s,\pi_s)\frac{\nabla\Phi^{\epsilon,\delta}}{|\nabla\Phi^{\epsilon,\delta}|}\cdot\frac{\nabla\Phi^{\epsilon,\delta}}{|\nabla\Phi^{\epsilon,\delta}|}   \right )(\bv\cdot\nabla\Phi^{\epsilon,\delta})
\\
&+ \sum_{i=1}^{d-1} \int_{\Omega}\left (\alpha_{BJ} \bu_s \cdot \tilde{\bftau}_i-\bfsigma(\nabla\bu_s,\pi_s)\tilde{\bftau}_i \cdot\tilde{\bftau}_i\right )(\bv\cdot \tilde{\bftau}_i ) |\nabla\Phi^{\epsilon,\delta}|.
\end{align*}

The first term has a negative sign and will be combined with the terms on  the left-hand side. The second term can be again estimated by \cite[Theorem 5.2]{burger2017analysis}:
\begin{multline*}
  \left |\int_{\Omega}\left (\nabla\cdot\bfsigma(\nabla\bu_s,\pi_s)\cdot\bv+\boldsymbol\kappa\Delta p_s \psi\right )(\chi-\Phi^{\epsilon,\delta})|
  \right |
  \ltt \\ (\epsilon^{3/2}+\sqrt{\delta})
  \left (                      
  \|\nabla\cdot\bfsigma(\nabla\bu_s,\pi_s)\|_{W^{1,\infty}(\Omega)^d}\|\bD(\bu_s-\bu_d^{\epsilon,\delta})\|_{L^2(\Omega,\Phi^{\epsilon,\delta})^{d \times d}} \right. \\
  \left. +\|\Delta p_s\|_{W^{1,\infty}(\Omega)}\|\nabla(p_s-p^{\epsilon,\delta}_d)\|_{L^2(\Omega,\Psi^{\epsilon,\delta})^d}
  \right ).
\end{multline*}

To estimate the last three terms first we notice that in the tubular neighborhood of $\Gamma$ we have $\bn=-\displaystyle\frac{\nabla\Phi^{\epsilon,\delta}}{|\nabla\Phi^{\epsilon,\delta}|}=-\displaystyle\frac{\nabla\Phi^{\epsilon}}{|\nabla\Phi^{\epsilon}|}$ and 
$\bftau_i={ -\tilde{\bftau}_i}, i=1,2,\ldots, d-1$. Therefore, using the coupling conditions, the following equalities hold on $\Gamma$:
\[
\left (\bu_s+\kappa\nabla p_s   \right )\cdot\nabla\Phi^{\epsilon,\delta}=0,
\
\left (-p_s-\bfsigma(\nabla\bu_s,\pi_s)\frac{\nabla\Phi^{\epsilon,\delta}}{|\nabla\Phi^{\epsilon,\delta}|}\cdot\frac{\nabla\Phi^{\epsilon,\delta}}{|\nabla\Phi^{\epsilon,\delta}|}   \right )=0, \
\left ( \alpha_{BJ}  \bu_s\cdot \tilde{\bftau}_i-\bfsigma(\nabla\bu_s,\pi_s)\tilde{\bftau}_i \cdot\tilde{\bftau}_i\right )=0.
\]
Now from \eqref{regularizacija} we have $\nabla\Phi^{\eps,\delta}=(1-2\delta)\nabla\Phi^{\eps}$.
Hence, we can use \cite[Theorem 5.6]{burger2017analysis} to bound the terms that are multiplied by $\nabla \Phi^\eps$,
and the trace inequality of Lemma \ref{nablaPhi}, to bound terms that have as a factor $\delta \nabla \Phi^\eps$. Combining these estimates, last three terms are bounded, up to a constant, by:
\[
(\epsilon^{3/2}+\delta)
\left (
\|\bu_s\|_{H^3(\Omega)^d}+\|\pi_s\|_{H^2(\Omega)}+\|p_s\|_{H^3(\Omega)}
\right )
\left (
\|\bD(\bu_s-\bu_d^{\epsilon,\delta})\|_{L^2(\Omega,\Phi^{\epsilon,\delta})^{d \times d}}
+\|\nabla(p_s-p^{\epsilon,\delta}_d)\|_{L^2(\Omega,\Psi^{\epsilon,\delta})^d}
\right ).
\]
Combining all the obtained estimates and using Young's inequality, we get the following error estimate:
\begin{multline*}
\frac{1}{2}\frac{d}{dt}\left ( 
\rho\|\bu_s-\bu_d^{\epsilon,\delta}\|_{L^2(\Omega,\Phi^{\epsilon,\delta})^d}^2
+c_0\|p_s-p^{\epsilon,\delta}_d)\|_{L^2(\Omega,\Psi^{\epsilon,\delta})}^2
\right )
+\mu \|\bD(\bu_s-\bu_d^{\epsilon,\delta})\|_{L^2(\Omega,\Phi^{\epsilon,\delta})^{d \times d}}^2
\\
+ \frac{\boldsymbol \kappa}{2}\|\nabla(p_s-p^{\epsilon,\delta}_d)\|_{L^2(\Omega,\Psi^{\epsilon,\delta})^d}^2
+\alpha_{BJ} \sum_{i=1}^{d-1} \int_{\Omega}\left |(\bu_s-\bu_d^{\epsilon,\delta})\cdot \tilde{\bftau}_i \right |^2 |\nabla \Phi^{\epsilon,\delta}|
\\
\ltt (\epsilon^3+\delta)
\left (
\|\partial_t\bu_s\|^2_{W^{1,\infty}(\Omega)^d}
+\|\partial_t p_s\|^2_{W^{1,\infty}(\Omega)}
+\|\bu_s\|^2_{W^{3,\infty}(\Omega)^d}+\|\pi_s\|^2_{W^{2,\infty}(\Omega)}+\|p_s\|^2_{W^{3,\infty}(\Omega)} \right. \\
\left.
+\|\bF\|_{H^1(\Omega)^d}+\|g\|_{H^1(\Omega)}
\right ).
\end{multline*}
The final result \eqref{ModellinEstimate} is obtained by integrating the estimate above  from   0 to $t$.
\end{proof}

\begin{remark}[Sharpness]
Notice that the order of convergence with respect to $\delta$ obtained in Theorem~\ref{ModellinError} is sharp. Namely, since we arbitrarily extended the sharp interface solution to the whole domain $\Omega$, we cannot expect that this extension will be close to the diffuse interface solution. Therefore the error estimates in the ``wrong'' part of domain are of order $\delta$. More precisely, we have
\[
  \|\bu_s-\bu_d^{\eps,\delta}\|_{L^2(\Omega,\Phi^{\eps,\delta})^d}^2
  \geq  \int_{\Omega\setminus\Omega^{\eps}_F}\delta (\bu_s-\bu_d^{\eps,\delta})^2
  \approx\delta.
\]
From a computational point of view this is fine since we typically take $\delta \ll \eps$. Moreover, note that Theorem \ref{ModellinError} also holds for $\delta=0$. However, $\Phi^{\eps,0}$ does not satisfy the assumptions used in Section~\ref{sub:PF}.
\end{remark}

\begin{remark}[Extensions]
At first sight it might seem counter-intuitive that the conclusion of Theorem \ref{ModellinError} holds for an arbitrary extension of the sharp interface solution. However, the constant in estimate \eqref{ModellinEstimate} depends on this extension. Since the same extension is chosen for every $\eps$ and $\delta$, the estimate presented in \eqref{ModellinError} is uniform in $\delta$ and $\epsilon$, and depends just on norms of the sharp interface solution.
\end{remark}


%

\subsection{Muckenhoupt weights}
 Here we estimate the difference between the continuous solutions of the diffuse domain  and sharp interface formulations in terms of the parameter $\eps$ describing the diffuse interface width, under the sole assumption on the phase field function that $\Phi^\eps \in A_2$. 


We begin with some preliminary considerations.

\begin{lemma}[Embeddings]\label{lem:EmbeddingsForModelError}
For every $\eps >0$ we have that
\begin{enumerate}[1.]
  \item The embedding $L^2(\Omega_F^\eps) \subset L^2(\Omega_F^\eps,\Phi^\eps)$ is continuous, independently of $\eps$.
  
  \item If $w \in L^2(\Omega_F^\eps,\Phi^\eps)$ then the restriction $w_{|\Omega_F} \in L^2(\Omega_F)$ with an embedding constant independent of $\eps$.
  
  \item Let $\Gamma_F^\eps = \partial\Omega_F^\eps \cap \Omega$, i.e., the diffuse boundary of the fluid domain. If $w \in H^1(\Omega_F^\eps,\Phi^\eps)$ then $w_{|\Gamma_F^\eps } \in L^2(\Gamma_F^\eps )$.
\end{enumerate}
Similar embeddings hold for $L^2(\Omega_D^\eps, \Psi^\eps )$.
\end{lemma}
\begin{proof}
We prove each statement separately.
\begin{enumerate}[1.]
  \item Let us recall that $\alpha >0$. As a consequence, we have that $\Phi^\eps \in L^\infty(\Omega_F^\eps)$ and, if $w \in L^2(\Omega_F^\eps)$ we can write
  \[
    \int_{\Omega_F^\eps} |w|^2 \Phi^\eps \leq \| \Phi^\eps \|_{L^\infty(\Omega_F^\eps)} \int_{\Omega_F^\eps} |w|^2.
  \]
  
  \item To show this result, we recall that, if $x \in \overline{\Omega_F}$, then $\Phi^\eps(x) \geq \tfrac1{2}$; see Figure~\ref{fig:DomainsAndLayers}. Thus, if $w \in L^2(\Omega_F^\eps,\Phi^\eps)$,
  \[
    \int_{\Omega_F} |w|^2 \leq 2 \int_{\Omega_F} |w|^2 \Phi^\eps \leq 2 \int_{\Omega_F^\eps} |w|^2 \Phi^\eps.
  \]

  \item This has already been mentioned before; see, \eqref{eq:TraceIsL2}.
\end{enumerate}
\end{proof}

The idea behind establishing a bound for the model error, i.e., the difference between the solution to the sharp interface and diffuse domain solutions is simple. We test one problem with the solution of the other and compare. The previous result gives us a framework that justifies this approach. Then, if we compare the differences over $\Omega_F$, these are controlled by integrals over small domain layers. Using \cite[Section 5]{burger2017analysis} as inspiration, we can establish an estimate by invoking smoothness and convergence of diffuse domain integrals.

We now implement this scheme. Let us begin by assuming that $\boldsymbol\kappa$ is constant, that there are stable extensions of all sharp interface functions to the corresponding diffuse domains, and that
\begin{align*}
  \partial_t \bu_s &\in L^2(0,T;L^2(\Omega_F)^d), & \partial_t \bu_d^\eps &\in L^2(0,T;L^2(\Omega_F^\eps,\Phi^\eps)^d), \\
  \partial_t p_s &\in L^2(0,T;L^2(\Omega_D)), & \partial_t p_d^\eps &\in L^2(0,T;L^2(\Omega_D^\eps, \Psi^\eps )).
\end{align*}
Owing to Lemma~\ref{lem:EmbeddingsForModelError} it is legitimate to set $(\bv,\psi,\varphi) = (\be_\bu,e_p,e_\pi) = (\bu_s-\bu_d^\eps,p_s-p_d^\eps,\pi_s-\pi_d^\eps)$ in \eqref{weakSI} and \eqref{SDFP_weak}, where we implicitly extend the sharp interface functions to the diffuse domain. We subtract the ensuing identities, and use that $\nabla\cdot \be_\bu = 0$ almost everywhere, to obtain
\begin{multline}
  \frac12 \frac{d}{dt} \left[ \rho \int_{\Omega_F} |\be_\bu|^2 \Phi^\eps + c_0 \int_{\Omega_D} |e_p|^2 \Psi^\eps  \right] + 2\mu \int_{\Omega_F} |\bD(\be_\bu)|^2 \Phi^\eps + \boldsymbol{\kappa}\int_{\Omega_D} |\nabla e_p|^2  \Psi^\eps  \\
  + \frac{\alpha_{BJ}}{2\eps} \sum_{i=1}^{d-1} \int_{\ell^\eps} (\be_\bu\cdot\tilde{\bftau}_i)^2 |\nabla \dist_\Gamma| 
  =  
  \sum_{j=1}^4 \cR_j^\bu +\sum_{j=1}^4 \cR_j^p + \sum_{j=1}^3 \cR_j^\Gamma + \sum_{j=1}^2 \cR_j^\bF + \sum_{j=1}^2 \cR_j^g,
  \label{eq:ModelErrorEst}
\end{multline}
where
\begin{align*}
  \cR_1^\bu &= \rho \int_{\ell_D^\eps} \partial_t \bu_s \cdot \be_\bu \Phi^\eps, & 
  \cR_2^\bu &= \rho \int_{\ell_F^\eps} \partial_t \bu_s \cdot \be_\bu (\Phi^\eps-1), \\
  \cR_3^\bu &= 2\mu \int_{\ell_D^\eps} \bD(\bu_s):\bD(\be_\bu) \Phi^\eps, & 
  \cR_4^\bu &= 2\mu\int_{\ell_F^\eps} \bD(\bu_s):\bD(\be_\bu)(\Phi^\eps-1), \\
  \cR_1^p &= c_0 \int_{\ell_F^\eps} \partial_t p_s e_p  \Psi^\eps , & 
  \cR_2^p &= c_0 \int_{\ell_D^\eps} \partial_t p_s e_p {(\Psi^\eps-1)}, \\
  \cR_3^p &= {\boldsymbol\kappa} \int_{\ell_F^\eps} \nabla p_s \cdot \nabla e_p  \Psi^\eps , & 
  \cR_4^p &= {\boldsymbol\kappa} \int_{\ell_D^\eps} \nabla p_s \cdot \nabla e_p {(\Psi^\eps-1)}, \\
  \cR_1^\Gamma &= \int_\Gamma e_p (\bu_s\cdot\bn) - \frac1{2\eps}\int_{\ell^\eps} e_p \bu_s \cdot  \nabla \dist_\Gamma, &
  \cR_2^\Gamma &= \frac1{2\eps} \int_{\ell^\eps} p_s \be_\bu\cdot \nabla \dist_\Gamma -\int_\Gamma p_s (\be_\bu\cdot\bn), \\
  \cR_3^\Gamma &= \alpha_{BJ} \sum_{i=1}^{d-1} \cR_{3,i}^\Gamma, &
  \cR_{3,i}^\Gamma &= \frac1{2\eps} \int_{\ell^\eps} (\bu_s\cdot\tilde\bftau_i)(\be_\bu\cdot\tilde\bftau_i)|\nabla\dist_\Gamma| - \int_\Gamma (\bu_s\cdot \bftau_i)(\be_\bu\cdot\bftau_i) \\
  \cR_1^\bF &= -\rho\int_{\ell_D^\eps} \bF^\eps \cdot \be_\bu \Phi^\eps, &
  \cR_2^\bF &= \rho\int_{\ell_F^\eps} \bF\cdot \be_\bu (1-\Phi^\eps), \\
  \cR_1^g &= -\int_{\ell_F^\eps} g^\eps e_p  \Psi^\eps , &
  \cR_2^g &= \int_{\ell_D^\eps} g e_p {(1-\Psi^\eps)},
\end{align*}
and we used that $\Phi^\eps \equiv 1$ on $\Omega_F \setminus \ell_F^\eps$, and $\Phi^\eps \equiv 0$ on $\Omega_D \setminus \ell_D^\eps$, and similar values for $\Psi^\eps$. 

It is now a matter of, using regularity assumptions, estimate the residual terms given above.

\begin{theorem}[Modeling error II]\label{thm:ModelError}
Assume that
\begin{align*}
  \partial_t \bu_s &\in L^\infty(0,T;L^\infty(\Omega_F)^d), &\bD(\bu_s) &\in L^\infty(0,T;L^\infty(\Omega_F)^{d \times d}), \\
  \partial_t p_s &\in L^\infty(0,T;L^\infty(\Omega_D)), &\nabla p_s &\in L^\infty(0,T;L^\infty(\Omega_D)^{d}), \\
  \bF &\in L^\infty(0,T;L^\infty(\Omega_F)^d), &g &\in L^\infty(0,T;L^\infty(\Omega_D)),
\end{align*}
along with their extensions to the corresponding diffuse domains. Assume, in addition, that both $\{\bftau_i\}_{i=1}^{d-1}$ and $\{\tilde{\bftau}_i\}_{i=1}^{d-1}$ have suitable extensions onto $\Omega_F^\eps$ that belong to $C^1(\overline{\Omega_F^\eps})$. Then, we have
\[
  \frac12 \frac{d}{dt} \left[ \rho \int_{\Omega_F} |\be_\bu|^2 \Phi^\eps + c_0 \int_{\Omega_D} |e_p|^2  \Psi^\eps  \right] + 2\mu \int_{\Omega_F} |\bD(\be_\bu)|^2 \Phi^\eps + \boldsymbol{\kappa}\int_{\Omega_D} |\nabla e_p|^2  \Psi^\eps  \lesssim \eps^{1/2},
\]
where the implicit constant depends on the smoothness assumptions, but not on $\eps$.
\end{theorem}
\begin{proof}
Using the assumed regularity, each one of the terms in \eqref{eq:ModelErrorEst} is estimated as follows. First, we use that $0\leq \Phi^\eps \leq 1$, that on $\ell_F^\eps$ we have $\Phi^\eps \geq \tfrac1{2}$, the Cauchy Schwartz, and the weighted Korn's inequality of \eqref{eq:wKorn1} to infer
\[
  \sum_{j=1}^4 \cR_j^\bu \leq C \int_{\ell^\eps}\left(  | \partial_t \bu_s |^2 + | \bD (\bu_s) |^2 \right) + \frac\mu{4} \int_{\Omega_F^\eps} |\bD(\be_\bu)|^2 \Phi^\eps \leq C \eps + \frac\mu{4} \int_{\Omega_F^\eps} |\bD(\be_\bu)|^2 \Phi^\eps,
\]
where in the last step we used \eqref{eq:MeasureOfLayer}. Similarly,
\begin{align*}
  \sum_{j=1}^4 \cR_j^p &\leq C \int_{\ell^\eps} \left( | \partial_t p_s |^2 + | \nabla p_s |^2 \right) + \frac{\boldsymbol\kappa}{4} \int_{\Omega_D^\eps} |\nabla e_p|^2  \Psi^\eps
    \leq C \eps + \frac{\boldsymbol\kappa}{4} \int_{\Omega_D^\eps} |\nabla e_p|^2 \Psi^\eps, \\
  \sum_{j=1}^2 \cR_j^\bF &\leq C \int_{\ell^\eps} | \bF^\eps |^2 + \frac{\mu}4\int_{\Omega_F^\eps} |\bD(\be_\bu)|^2 \Phi^\eps
    \leq C \eps + \frac{\mu}4\int_{\Omega_F^\eps} |\bD(\be_\bu)|^2 \Phi^\eps, \\
  \sum_{j=1}^2 \cR_j^g &\leq C \int_{\ell^\eps} | g^\eps |^2 + \frac{\boldsymbol\kappa}4\int_{\Omega_D^\eps} |\nabla e_p|^2 \Psi^\eps
    \leq C \eps + \frac{\boldsymbol\kappa}4\int_{\Omega_D^\eps} |\nabla e_p|^2 \Psi^\eps
\end{align*}

We now estimate the diffuse interface terms. First we observe that, after implicitly extending to diffuse domains, we have $e_p \in H^1(\Omega_D^\eps,  \Psi^\eps )$ and $\bu_s \in W^{1,\infty}(\Omega_F)^d$. This implies that $e_p \bu_s \in \mathcal{X}^\eps$. This, in particular, implies that
\[
  \int_\Gamma e_p (\bu_s\cdot\bn) = \int_{\Omega_D} \nabla \cdot \left( e_p \bu_s \right).
\]
Notice now that, with the notation used in the proof of Lemma~\ref{nablaPhi}, we have
\[
  - \frac1{2\eps}\int_{\ell^\eps} e_p \bu_s \cdot  \nabla \dist_\Gamma = \int_{\Omega_D^\eps} e_p \bu_s \cdot  \nabla (1-\omega^\eps) = - \int_{\Omega_D^\eps} \nabla \cdot \left( e_p \bu_s \right)(1-\omega^\eps).
\]
Therefore, arguing as in \cite[Lemma 5.4]{burger2017analysis} yields
\[
  \cR_1^\Gamma = \int_{\Omega_D} \nabla \cdot \left( e_p \bu_s \right) - \int_{\Omega_D^\eps} \nabla \cdot \left( e_p \bu_s \right)(1-\omega^\eps) \lesssim \eps^{1/2}.
\]
A similar reasoning also gives $\cR_2^\Gamma \lesssim \eps^{1/2}$.

It remains then to estimate $\cR_3^\Gamma$. Under the assumption that the tangents have $C^1$ extensions, each of the terms that define $\cR_3^\Gamma$ fits into the setting of \cite[Lemma 5.4]{burger2017analysis}. This finally gives
\[
  \cR_3^\Gamma \lesssim \eps^{1/2},
\]
and proves the result.
\end{proof}

As a final comment we mention that, if we assume further regularity on the sharp interface solutions, a better order of convergence can be obtained. More precisely, let $\alpha$ be the exponent from the definition of the phase field function $\Phi^{\eps}$, see \eqref{eq:SfunctionBurgers}. Then under suitable regularity assumptions on the sharp interface solution, one can prove that the convergence rate is $\frac{3}{2}-(1-\alpha)$, i.e., arbitrarily close to the convergence rate from Theorem \ref{ModellinError} by a suitably choice of the phase field function. The proof is analogous to the proof of Theorem \ref{ModellinError} with one important difference.  Namely, the proof of Theorem \ref{ModellinError} relies on the results from \cite[Section 5]{burger2017analysis} which assume that the phase field function is Lipschitz and the function $S$ used \eqref{eq:defofPhiPsi}  is concave on $(0,1)$. Since both of these conditions are not satisfied with our choice of the phase field function (they are not compatible with the condition $\Phi^{\eps}\in A_2$), we cannot use the results from \cite[Section 5]{burger2017analysis} directly. However, by introducing suitable modification in their proofs, one can prove analogous convergence results for diffuse integrals that are valid also for a broader class of weights that also includes our choice of phase field function. The only trade-off is a loss of convergence of order $(1-\alpha)$. For brevity, we will not pursue this here.

\section{Numerical results}\label{Sec:numerics}

In this section, we illustrate the accuracy of our diffuse interface approach and further explore its capabilities with a series of numerical illustrations. All of our computations were carried out using the finite element solver \emph{FreeFem++}~\cite{hecht2012new}. 

We must comment that, for implementation reasons, all the terms in the numerical scheme are formulated over the entire domain $\Omega$. Furthermore, in all cases we considered, we must regularize the phase field functions $\Phi^\eps$ and $\Psi^\eps$ as described in \eqref{regularizacija}. Otherwise, the ensuing system matrices become singular, as all the degrees of freedom that belong to one diffuse subdomain but not the other would have a zero row in these matrices. The value of the regularization parameter $\delta$ is indicated in each example. The numerical method used in this section is summarized as follows: For $n = 0, \ldots, N-1$, find $(\bu_h^{n+1}, p_h^{n+1}, \theta_h^{n+1} )$ such that, for every $(\bv_{h},\psi_h,\varphi_h)$, we have:
\begin{equation}
  \begin{aligned}
	\rho\int_{\Omega}\partial_t \bu^{n+1}_h \cdot\bv_h \Phi^{\epsilon}
	+c_0\int_{\Omega} \partial_t p^{n+1}_h \psi_h \Psi^{\epsilon}  
	+2\mu\int_{\Omega}\bD(\bu^{n+1}_h):\bD(\bv_h)\Phi^\eps
	 -\int_{\Omega} \nabla \cdot \bv_h \pi^{n+1} \Phi^\eps
	 +\int_{\Omega} \nabla \cdot \bu_h^{n+1} \varphi_h \Phi^\eps
	\\ 
	+\int_{\Omega} \boldsymbol\kappa\nabla p^{n+1}_h \cdot \nabla \psi_h \Psi^{\epsilon}
	-\int_{\Omega} p^{n+1}_h \bv_h \cdot \nabla\Phi^{\epsilon}
	+\int_{\Omega} \psi_h \bu^{n+1}_h \cdot \nabla\Phi^{\epsilon}
	+\alpha_{BJ} \sum_{i=0}^{d-1}\int_{\Omega}(\bu^{n+1}_h \cdot \tilde{\bftau}_i)(\bv_h\cdot \tilde{\bftau}_i) |\nabla\Phi^{\epsilon}| 
 		\\ 
	=\int_{\Omega}\bF^{n+1}\cdot\bv_h\Phi^{\epsilon}+\int_{\Omega}g^{n+1}\psi \Psi^{\eps}.
  \end{aligned}
 \label{comp_method}
\end{equation}

\subsection{Rates of convergence}\label{sub:rates}

To test the rates of convergence we proved in previous sections, we use the method of manufactured solutions on a benchmark problem previously used in~\cite{layton2012long}. We define the Stokes domain as $\Omega_F=(0,1)\times (0,1)$ and the Darcy domain as $\Omega_D=(0,1) \times (1,2)$. Therefore, the computational domain used in this example is $\Omega=\Omega_F \cup \Omega_D = (0,1) \times (0,2)$, and the interface is $\Gamma = (0,1) \times \{1\}$. The phase field function, in this example, is defined as
\[
  \Phi^{\epsilon} = \frac12 \left(1+\tanh \left(\frac{y-1}{\eps} \right) \right),
\]
and regularized as in \eqref{regularizacija}. Notice that this weight is smooth near the interface.

The exact, manufactured, solution is
\begin{align*}
  \bu^{exact}(x,y,t) &= \left(-\frac{1}{\hat{\pi}} e^y \sin(\hat{\pi} x) \cos(2 \hat{\pi} t), \;
                          \left( e^y -e \right) \cos(\hat{\pi} x) \cos(2 \hat{\pi} t) \right)^\intercal,\\
  \pi^{exact}(x,y,t) &= 2 e^y \cos(\hat{\pi} x) \cos(2 \hat{\pi} t), \\
  p^{exact}(x,y,t) &= \left( e^y -e y \right) \cos(\hat{\pi} x) \cos(2 \hat{\pi} t),
\end{align*}
where $\hat\pi=3.14159\ldots$ is the universal constant equal to the ratio between the circumference and the diameter of the unit disk. On the bottom boundary, we impose Dirichlet boundary conditions for the Darcy pressure, and on the top boundary we impose Dirichlet conditions for the Stokes velocity. On the left and right boundaries, we impose Neumann conditions for both Stokes and Darcy regions. All the boundary conditions are computed using the exact solution. We set the physical parameters in this example to be $\rho=\nu=c_0=\alpha_{BJ}=1$, $\boldsymbol \kappa =\mathbb{I}$, where $\mathbb{I}$ is the identity matrix, and $T=1$.

The Neumann-type boundary conditions are imposed using the diffuse interface approach as follows:
\begin{align*}
  \int_{\Gamma^F_{\textrm{left/right}}} \bfsigma  \bn \cdot \bv&= \int_{\Gamma_{\textrm{left/right}}} \bfsigma  \bn \cdot \bv \Phi^{\eps},
\\
\int_{\Gamma^D_{\textrm{left/right}}} \boldsymbol\kappa \nabla p \bn \psi &= \int_{\Gamma_{\textrm{left/right}}} \boldsymbol\kappa \nabla p \bn \psi  \Psi^\eps ,
\end{align*}
where $\Gamma^F_{\textrm{left/right}}$ and $\Gamma^D_{\textrm{left/right}}$ denote the left and right boundaries to $\Omega_F$ and $\Omega_D$, respectively, and $\Gamma_{\textrm{left/right}}$ denotes the left and right boundary to $\Omega$. We use $\mathbb{P}_2-\mathbb{P}_1$ elements for the Stokes velocity and pressure, and $\mathbb{P}_1$ elements for the Darcy pressure. We initially set $\Delta t=h=\tfrac15$, $\eps =h$, and $\delta = 10^{-3}$. These parameters are then refined by halving them. The total velocity and pressure are defined, respectively, by
\[
  \bu_{tot} = \bu   \Phi^{\eps}  +\bq  \Psi^{\eps}, \quad p_{tot} =\pi   \Phi^{\eps}  +p \Psi^{\eps},
\]
with their aid we define the relative errors of our scheme to be
\[
  \mathbf{e}_u = \frac{\|\bu_{tot}^{exact} - \bu_{tot}^{\eps}  \|_{L^2(\Omega;\R^2)}}{\| \bu_{tot}^{exact} \|_{L^2(\Omega)^2}},
\qquad
  p_{tot} =\pi   \Phi^{\eps}  +p(1-  \Phi^{\eps}), \quad  {e}_p = \frac{\|p^{exact}_{tot} - p^{\eps}_{tot} \|_{L^2(\Omega)}}{\| p^{exact}_{tot}   \|_{L^2(\Omega)}}.
\]
evaluated at the final time, $T$. In addition to the method based on the backward Euler time discretization described in~\eqref{comp_method}, we also consider a method based on the midpoint scheme. The solution to the midpoint scheme is obtained by solving~\eqref{comp_method} over a half time interval (using $\Delta t/2$), resulting in $(\bu^{n+\frac12}, p^{n+\frac12}, \theta^{n+\frac12})$, and then extrapolating the solution as:
\[
\bu^{n+1} = 2 \bu^{n+\frac12} - \bu^n, \quad
p^{n+1} = 2 p^{n+\frac12} - p^n, \quad
\pi^{n+1} = 2 \pi^{n+\frac12} - \pi^n,
\]
at each time step, as described in~\cite{burkardt2020refactorization}.

\begin{table}[ht]
\begin{center}
 \begin{tabular}{c | c c | c c || c c | c c} 
   &  \multicolumn{4}{c}{backward Euler} & \multicolumn{4}{c}{midpoint}\\
 $h$ & $\mathbf{e}_u$ & rate & $e_p$ & rate & $\mathbf{e}_u$ & rate & $e_p$ & rate
 \\
  \hline  
 $\frac15$ & $3.96\cdot 10^{-1} $ & - & $4.69 \cdot 10^{-1}$  & - & $7.84 \cdot 10^{-1}$ & - &  $1.83 \cdot 10^{0}$ & -
 \\
 $\frac1{10}$ & $9.41\cdot 10^{-2}$ & 2.07 & $1.10 \cdot 10^{-1}$ & 2.09 & $1.17 \cdot 10^{-1}$ & 2.75 &  $1.57  \cdot 10^{-1}$ & 3.55
  \\
 $\frac1{20}$ & $4.06\cdot 10^{-2}$ & 1.21 & $4.80  \cdot 10^{-2}$ & 1.20  & $3.05  \cdot 10^{-2}$  & 1.94 & $3.14  \cdot 10^{-2}$ & 2.32
   \\
 $\frac1{40}$ & $1.87 \cdot 10^{-2}$ & 1.12 & $2.27 \cdot 10^{-2}$ & 1.08 & $9.58  \cdot 10^{-3}$  & 1.67 &  $ 6.81  \cdot 10^{-3}$  & 2.21
    \\
 $\frac1{80}$ & $8.90 \cdot 10^{-3}$ & 1.07 & $1.11  \cdot 10^{-2}$ & 1.03 & $3.36  \cdot 10^{-3}$  & 1.51 & $1.88  \cdot 10^{-3}$ & 1.86
 \end{tabular}       
\end{center}
\caption{Rates of convergence for the total velocity and pressure for the problem of Section~\ref{sub:rates}. The results were obtained with $\mathbb{P}_2-\mathbb{P}_1$ elements for the Stokes velocity and pressure and $\mathbb{P}_2$ elements for the Darcy pressure. We set $\Delta t=h=\epsilon$. The regularization parameter, $\delta$, is initially set to $10^{-3}$, and then refined at the same rate as $h$.}
\label{ratesP1}
\end{table}

Table~\ref{ratesP1} shows the rates of convergence, which agree with the theory. In particular, the first-order convergence is obtained when the backward Euler scheme is used for time discretization, while a loss of a half an order is observed when the midpoint method is used, agreeing with the estimates in~\eqref{ModellinEstimate} which indicate that the dominant error in this case would be due to term $\mathcal{O}(\eps^{\frac32})$. 

\begin{figure}[ht]
  \begin{center}
    \includegraphics[scale=0.6]{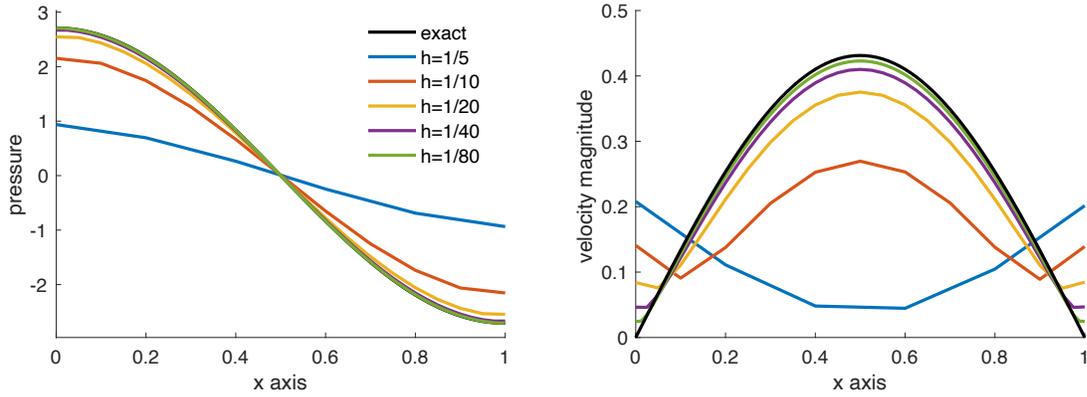}
  \end{center}
\caption{The total pressure (left) and the total velocity magnitude (right) over the interface $\Gamma = (0,1) \times \{1\}$ obtained for different values of $h$ for the problem of Section~\ref{sub:rates}.}
\label{intplot}
\end{figure}
We also plot, in Figure~\ref{intplot}, the Darcy pressure and the fluid velocity over the interface $\Gamma$ for different values of $h$.  We observe that the approximate solutions converge towards the exact ones as $h$ decreases. A more accurate approximation can be obtained if the meshes are refined near the interface, as shown in Example~\ref{sub:fun} below.

\subsection{An interesting interface}\label{sub:fun}

We now consider a benchmark problem with a curved boundary. We define the domain to be $\Omega = (0,1) \times (-1,1)$: see Figure~\ref{domainEx2}. 
\begin{figure}[ht]
  \begin{center}
    \begin{tikzpicture}
      \draw plot[domain=0:3,smooth] (\x, {0.3*sin(4*pi* \x r/3)});
      \draw plot[domain=0:3,smooth] (\x, {-3});
      \draw plot[domain=0:3,smooth] (\x, {3});
      \draw plot[domain=-3:3,smooth] (0, {\x});
      \draw plot[domain=-3:3,smooth] (3, {\x});
      \coordinate (A) at (-0.5,0.5);
      \coordinate (B) at (3.5,0.5);
      \coordinate (C) at (1.5,1.4);
      \coordinate (D) at (1.2,-0.9);
      \coordinate (E) at (-0.5,-2.5);
      \coordinate (F) at (3.5,-2.5);
      \coordinate (G) at (1.5,-3.8);
      \coordinate (I) at (1.5,0.3);
      \coordinate (J) at (1.5,-2.5);
      \node[yshift=1.2cm] at (A) {$\Gamma_D^2$};
      \node[yshift=1.2cm] at (B) {$\Gamma_D^2$};
      \node[yshift=1.2cm] at (C) {$\Gamma_D^1$};
      \node[yshift=1.2cm] at (D) {$\Gamma$};
      \node[yshift=1.2cm] at (E) {$\Gamma_F^{in}$};
      \node[yshift=1.2cm] at (F) {$\Gamma_F^2$};
      \node[yshift=1.2cm] at (G) {$\Gamma_F^1$};
      \node[yshift=1.2cm] at (I) {$\Omega_D$};
      \node[yshift=1.2cm] at (J) {$\Omega_F$};
    \end{tikzpicture}
  \end{center}
  \caption{The domain used in the example of Section~\ref{sub:fun}.} \label{domainEx2}
\end{figure}
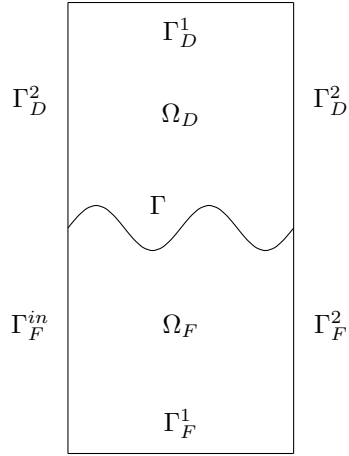
The interface, $\Gamma$, is given by 
\[
  \Gamma = \{ (x,y)  \in \mathbb{R}^2\; : \; x \in (0,1),  \; y = 0.1 \sin(4\pi x)\}.  
\]
The top part of the domain represents the Darcy region, whereas the bottom part corresponds to the Stokes domain. The boundary conditions are as in \eqref{StokesBC}--\eqref{DarcyBC}, except on $\Gamma_F^{in}$, the inflow boundary, where we prescribe
\begin{align*}
  \bu =\left( -u_{in} \frac{y (1+y)}{0.5^2}, \; 0 \right)^\intercal  \quad \textrm{on } \Gamma_F^{in} \times (0,T),
\end{align*}
with $u_{in} = 10$. The physical parameters used in this example are given in Table~\ref{example2Par}. 

\begin{table}[ht]
  \begin{tabular}{l l | l l }
    \textbf{Parameters} & \textbf{Values} & \textbf{Parameters} & \textbf{Values}  \\
    \hline
    \hline
    \textbf{Fluid density} $\rho_f$ (g/cm$^3$)& $1$ &\textbf{Dynamic viscosity} $\mu$ (poise) & $0.035$    \\
    \textbf{Storativity coeff.} $c_0$ (cm$^2$/dyne) & $10^{-3}$ & \textbf{Hydraulic conductivity} $\boldsymbol \kappa$ (cm$^3$ s/g) & $10^{-5} \mathbb{I}$     \\
     \textbf{Slip rate} $\gamma $   (g/cm$^2$ s)& $10^3$    & \textbf{Final time} $T$ (s) & 3\\
  \end{tabular}
\caption{The  parameters used for the problem of Section~\ref{sub:fun}.}
\label{example2Par}
\end{table}

We solve this problem using both sharp and diffuse interface formulations. In both cases, we use MINI elements for the Stokes velocity and pressure and $\mathbb{P}_1$ elements for the Darcy pressure. The Darcy velocity is obtained from the pressure by post-processing.  In the diffuse interface formulation, the phase-field function $\Phi^{\eps}$ is defined as
\[
  \Phi^{\eps} = \frac12 \left(1+\tanh\left(\frac{-y+0.1 \sin(4 \pi x)}{\eps}\right) \right),
\] 
and regularized as described in \eqref{regularizacija}, with regularization parameter $\delta = 10^{-3}$. In both cases, the time step is taken to be $\Delta t =10^{-2}$. The sharp interface problem is solved on a mesh containing 3542 elements, labeled as Mesh 1 in Figure~\ref{ex2meshes}. For the diffuse interface configuration, we consider three unstructured meshes: Meshes 2--4. Mesh 2 and Mesh 3 consist of 12281 and 4129 elements, respectively, and are both refined close to the interface. Mesh 4 is a structured mesh consisting of 7200 elements; see~Figure~\ref{ex2meshes}. 
\begin{figure}[ht]
  \begin{center}
    \includegraphics[scale=1.1]{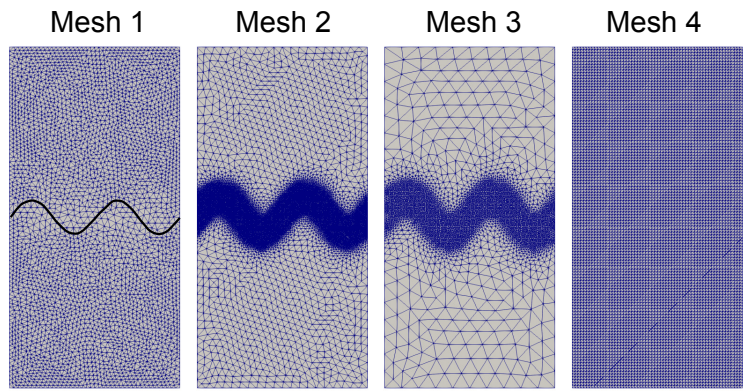}
  \end{center}
\caption{Computational meshes used in the problem of Section~\ref{sub:fun}.  Mesh 1  is used for the sharp interface problem. Meshes 2, 3 and 4  are used for the diffuse interface problem.}
\label{ex2meshes}
\end{figure}

\begin{figure}[ht]
  \begin{center}
    \includegraphics[scale=0.2]{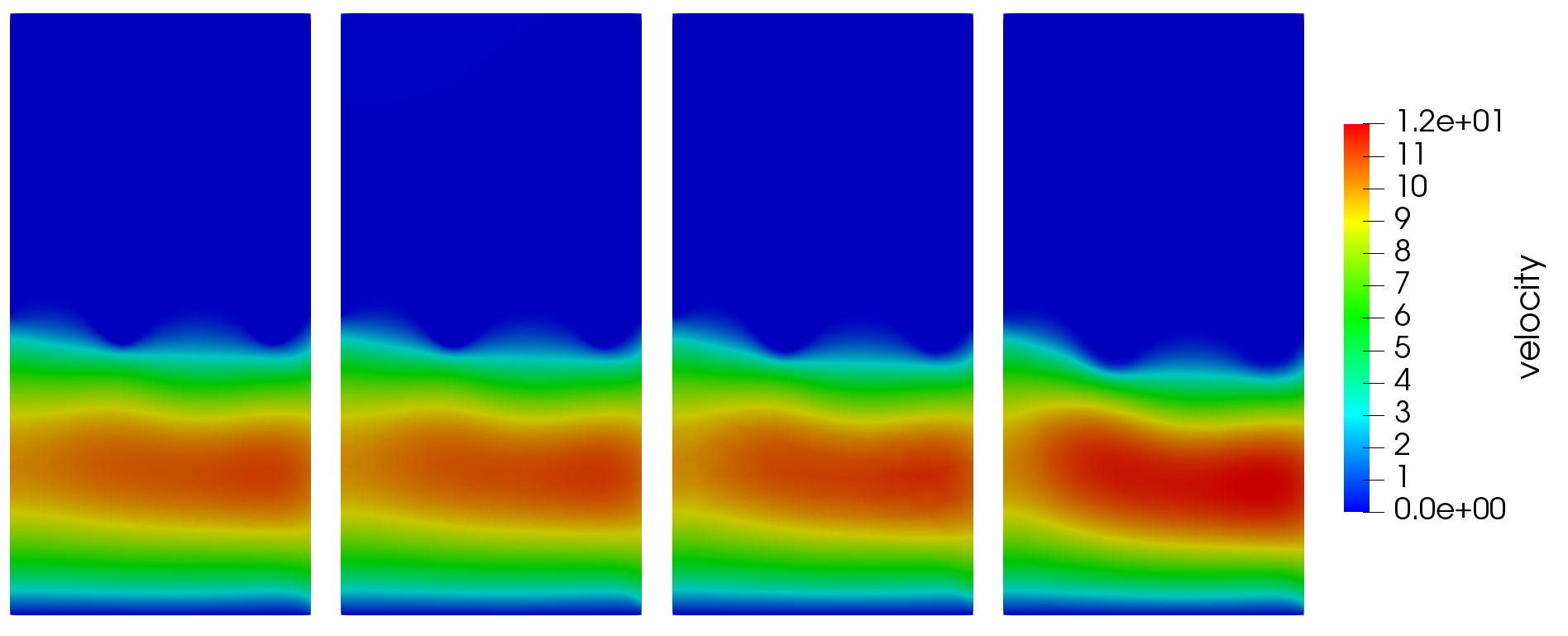}
  \end{center}
\caption{Results for the problem of Section~\ref{sub:fun}. Stokes and Darcy velocities obtained using the sharp interface formulation on Mesh 1 (left). The total velocity obtained using the diffuse interface formulation on Meshes 2 (middle-left), 3 (middle-right) and 4 (right).}
\label{ex2velocity}
\end{figure}
Figure~\ref{ex2velocity} shows the velocity obtained using the sharp interface model (leftmost panel), and the diffuse interface model on Mesh 2 (middle-left panel), Mesh 3 (middle-right panel) and Mesh 4 (rightmost panel). The flow enters the Stokes region from the left and is parallel to the Darcy region. Due to the small hydraulic conductivity coefficient, the flow in the Darcy region is much smaller than the flow in the Stokes region. We can observe that the diffuse interface formulation gives a good approximation of the velocity when the problem is solved using meshes which are refined close to the interface, while small differences are seen when a uniform mesh is used.

\begin{figure}[ht]
  \begin{center}
    \includegraphics[scale=0.2]{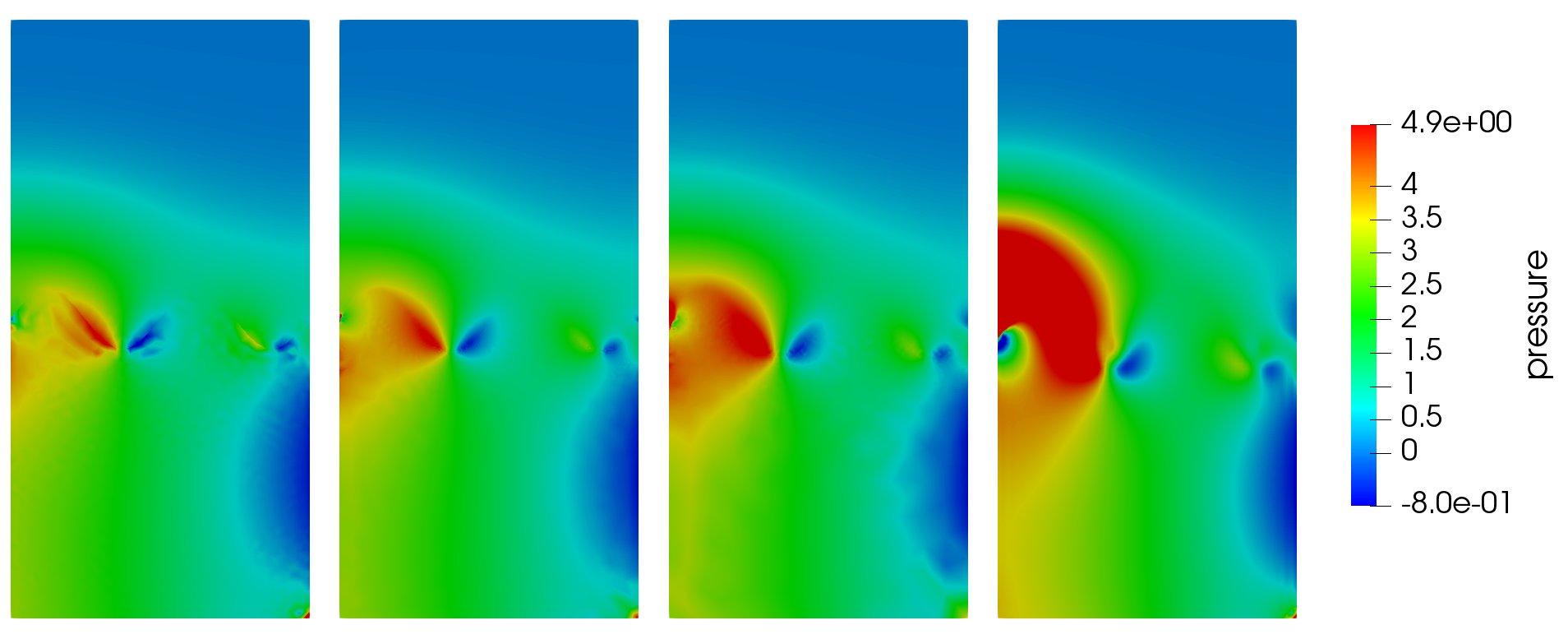}    
  \end{center}
\caption{Results for the problem of Section~\ref{sub:fun}. Stokes and Darcy pressures obtained using the sharp interface formulation on Mesh 1 (left). The total pressure obtained using the diffuse interface formulation on Meshes 2 (middle-left), 3 (middle-right) and 4 (right).}
\label{ex2pressure}
\end{figure}

The Stokes and the  Darcy pressures are shown in Figure~\ref{ex2pressure}. As before, the four panels represent the solution obtained on the four different meshes shown in Figure~\ref{ex2meshes}, where the first mesh is used for the sharp interface method, while the other meshes are used for the diffuse interface method. We observe that the pressure shows more sensitivity to the mesh than the velocity. The solution obtained using the diffuse interface approach seems to much better approximate the sharp interface solution when the mesh is locally refined towards to the interface. 
\begin{figure}[ht]
  \begin{center}
    \includegraphics[scale=0.9]{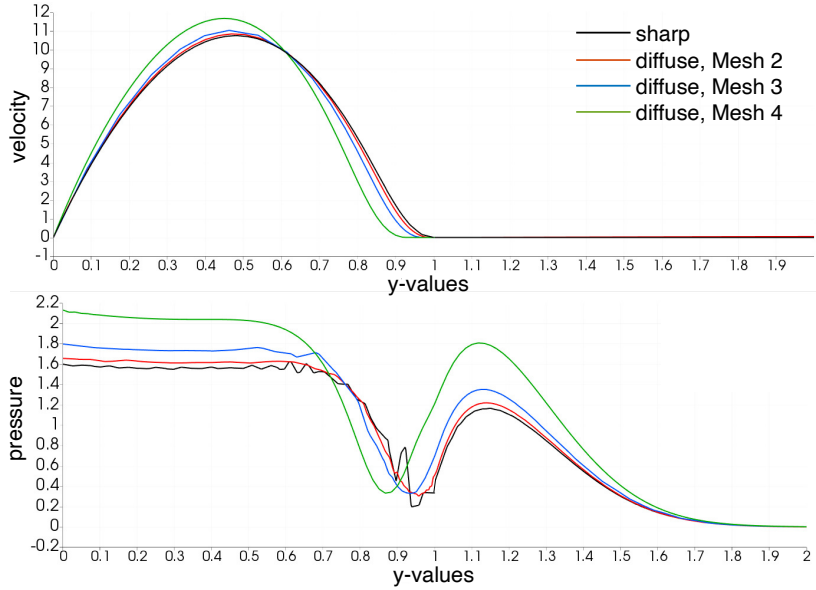}
  \end{center}
\caption{Results for the problem of Section~\ref{sub:fun}. The velocity (top) and pressure (bottom) obtained using the sharp interface method on Mesh 1, and the diffuse interface method on Meshes 2,3, and 4, plotted over the line $x=0.5$. }
\label{ex2LP}
\end{figure}

In order to additionally visualize the solution, we plot the velocity and pressure along a vertical line in the center of the domain,  $x=0.5$.
The one-dimensional plots obtained this way are shown in Figure~\ref{ex2LP}. As observed in Figure~\ref{ex2velocity}, the diffuse interface solution gives a good approximation of the velocity when Meshes 2 and 3 are used. Regarding the pressure, we see that the diffuse interface solution approaches the sharp solution as the mesh is refined. However, we also note that the sharp interface solution exhibits oscillations, particularly in the Stokes region, while the diffuse interface solution remains mostly smooth.

\subsection{Flow through a hexagonal network}\label{sub:hexagon}

This example focuses on modeling the flow in a hexagonal network, motivated by the flow in the lymph capillary network in a mouse tail (see Figure~\ref{ex3mt}).  
\begin{figure}[ht]
  \begin{center}
    \includegraphics[scale=3.3]{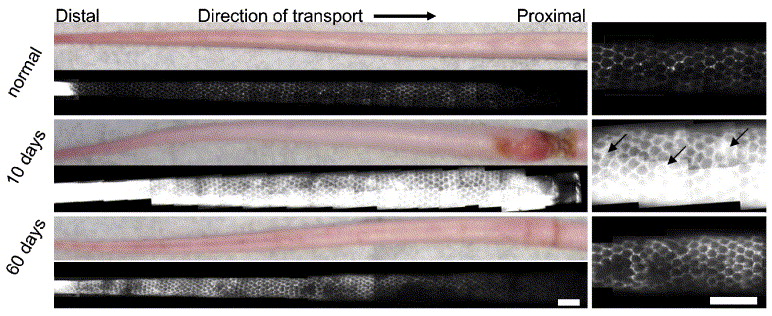}
  \end{center}
\caption{ A well-defined hexagonal network of dermal lymphatics that transport a fluorescent tracer in a mouse tail skin obtained in~\cite{rutkowski2006secondary}. The top panel shows the normal mouse tail skin.  At 10 days of lymphedema (shown in middle panel), the tail was swollen and the fluorescent tracer filled both the lymphatic vessels and the interstitial space (back flow indicated by arrows). At 60 days (bottom panel), lymphatic transport was improved but not fully restored. Reproduced with permission.}
\label{ex3mt}
\end{figure}
The computational domain is a rectangle with length 0.0696 cm and height 0.058 cm.   The computational mesh and the phase-field function for this problem are shown in Figure~\ref{ex3pf}. The mesh consists of 199,511 nodes and 398,491 elements.
\begin{figure}[ht]
  \begin{center}
    \includegraphics[scale=0.6]{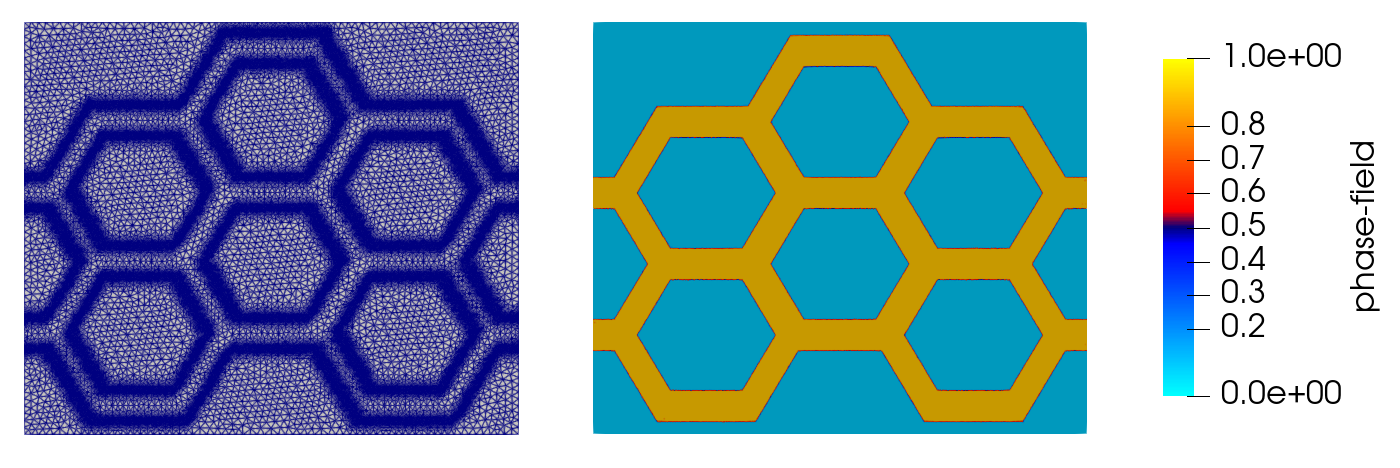}
  \end{center}
\caption{Computational mesh and phase field function for the problem of Section~\ref{sub:hexagon}.}
\label{ex3pf}
\end{figure}

\begin{table}[ht!]
\begin{tabular}{l l | l l }
\textbf{Parameters} & \textbf{Values} & \textbf{Parameters} & \textbf{Values}  \\
\hline
\hline
\textbf{Fluid density} $\rho_f$ (g/cm$^3$)& $1$ &\textbf{Dynamic viscosity} $\mu$ (poise) & $0.015$    \\
\textbf{Storativity coeff.} $c_0$ (cm$^2$/dyne) & $10^{-5}$ & \textbf{Permeability} $\boldsymbol \kappa$ (cm$^3$ s/g) & $4 \times 10^{-10} \mathbb{I}$     \\
 \textbf{Slip rate} $\gamma $   (g/cm$^2$ s)& $10^3$   \\
\end{tabular}
\caption{The  parameters used in the problem of Section~\ref{sub:hexagon}.}
\label{example3Par}
\end{table}

As before, the phase-field function equals one in the Stokes region, and zero in the Darcy region. The hexagonal network corresponds to the lymph capillary network in a mouse tail, where the flow is modeled using the Stokes equations. The Darcy region corresponds to the interstitium. We use $\mathbb{P}_2-\mathbb{P}_1$ elements for the Stokes velocity and pressure and $\mathbb{P}_1$ elements for the Darcy pressure. The parameters for this problem are taken from~\cite{leu1994flow,jayathungage2020investigations}, and are summarized in Table~\ref{example3Par}. 

The left channels of the Stokes region are inflow. On each one of them we impose a parabolic velocity profile with maximum velocity of $u_{in}=10^{-4}$, which is consistent with the experimental measurements of the lymphatic capillary flow velocity, see \cite{aukland1993interstitial,roose2012multiscale,wiig2012interstitial}.
The steady-state flow conditions are justified experimentally in \cite{leu1994flow,swartz1999mechanics,happel2012low}.
More precisely, we impose
\[
  \bu = - u_{in} \begin{dcases}
\left( \displaystyle\frac{(y+R)(y-R)}{R^2}, \; 0 \right)^\intercal,  &  -R \leq y \leq R, \\
   \left( \displaystyle\frac{(y+R+s)(y-R+s)}{R^2}, \; 0 \right)^\intercal,  &  -R-s \leq y \leq R-s,
  \end{dcases}
\]
where $R=2 \cdot 10^{-3}$ and $s=2 \cdot 10^{-2}$. One the right boundary of the Stokes domain  we impose zero normal stress for the fluid. For the Darcy problem, we impose $p=0$ on the entire boundary. The problem is solved until a steady state is reached, using the time step of $\Delta t=10^{-3}$. 

\begin{figure}[ht]
  \begin{center}
    \includegraphics[scale=0.6]{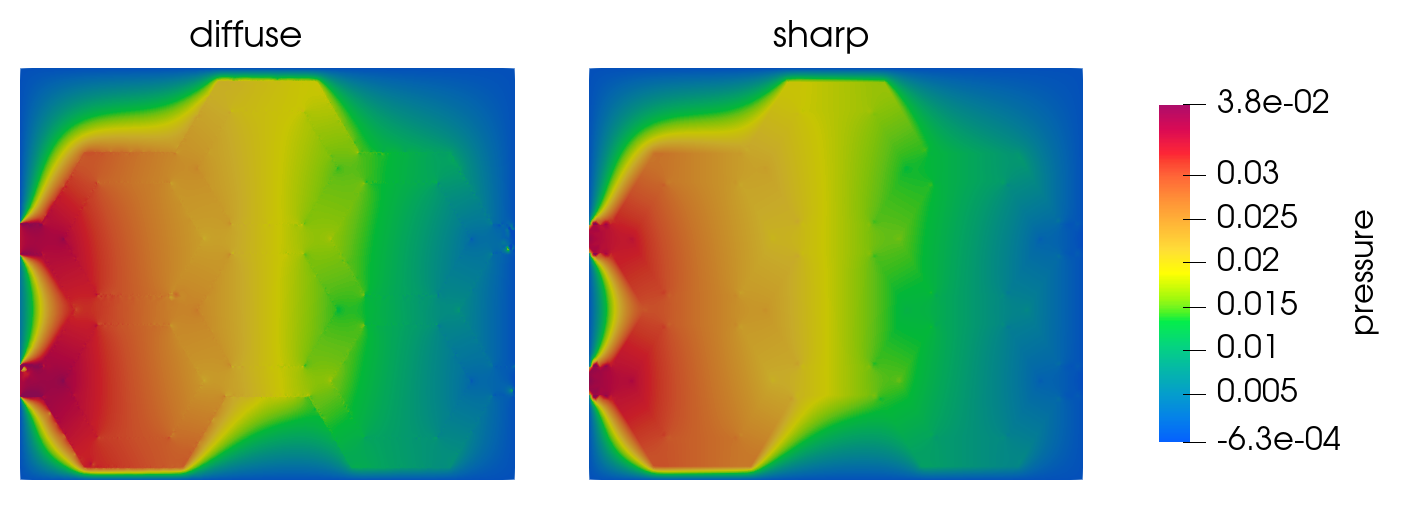}    
  \end{center}
\caption{Results for the problem of Section~\ref{sub:hexagon} at the final time. Left: Total pressure using a diffuse interface. Right: Stokes and Darcy pressures obtained using the sharp interface approach.}
\label{ex3press}
\end{figure}

To regularize the phase field function, we used $\delta=10^{-1}$. We compare the solutions obtained using the diffuse and sharp interface approaches. In Figure~\ref{ex3press} we compare the total pressure obtained using the diffuse interface method with the Stokes and Darcy pressures obtained using the sharp interface approach. An excellent agreement is observed. The velocities are compared in Figure~\ref{ex3vel}. Once again, the diffuse interface solution agrees very well with the sharp interface one. 

\begin{figure}[ht]
  \begin{center}
    \includegraphics[scale=0.6]{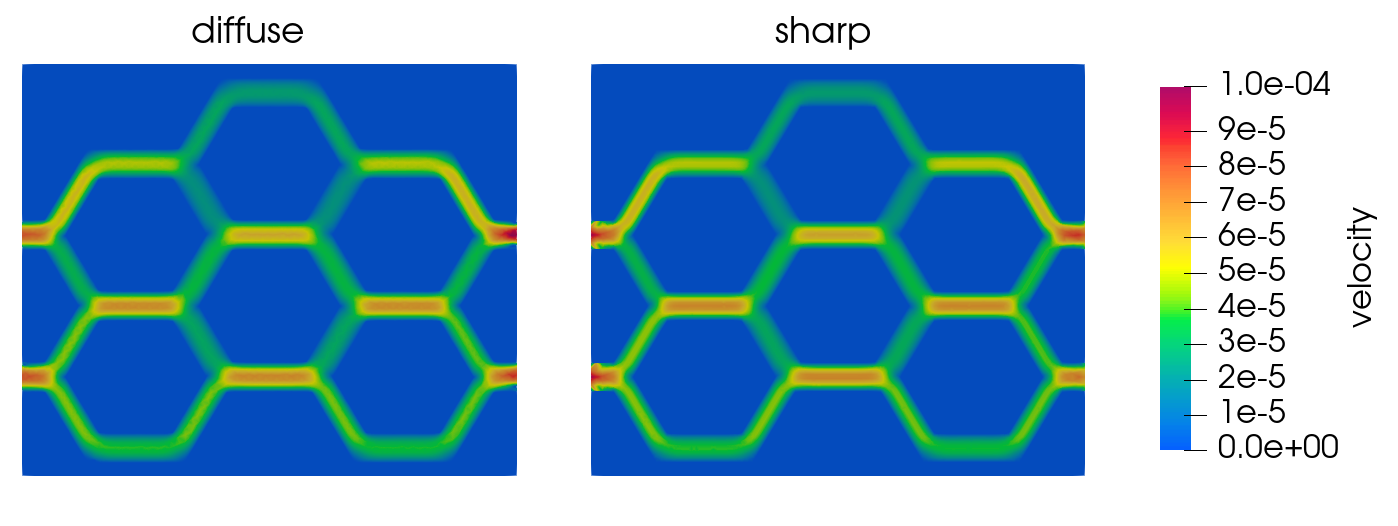}    
  \end{center}
\caption{Results for the problem of Section~\ref{sub:hexagon} at the final time. Left: Total velocity obtained using a diffuse interface. Right: Stokes and Darcy velocity obtained using the sharp interface approach. }
\label{ex3vel}
\end{figure}

\section{Conclusions}

In this work, we  considered the  time dependent Stokes-Darcy coupled problem formulated  using a diffuse interface approach. Our main focus was the convergence analysis of the diffuse interface formulation to the sharp interface one. This analysis consisted of two stages. We separately studied the modeling error resulting from the difference between the sharp and diffuse interface formulations at the continuous level, and the approximation error obtained when analyzing the difference between the continuous and discrete diffuse domain formulations.

Under suitable smoothness assumptions we obtained that the total (i.e., modeling plus discretization) error is of order
\[
  \mathcal{O}( \eps^{1/4} + \Delta t + h ),
\]
in the case the phase field function is a Muckenhoupt weight. We also presented some numerical simulations that show that higher order convergence rates are possible. Motivated by these results we also consider the case of a Lipschitz weight that is regularized with parameter $\delta$, and obtained
\[
  \mathcal{O}( \eps^{3/2} + \delta^{1/2} + \Delta t + h ).
\]
In this case, the numerical results agree with the theoretical predictions. Applications of the method to a problem with a curved interface, and a problem motivated by a lymphatic flow in a mouse tail have been explored numerically, demonstrating  the flexibility of the diffuse interface approach. 

\section{Acknowledgements }

Martina Buka\v{c} is partially supported by the National Science Foundation under grants  DMS-2208219, DMS-2205695, DMS-1912908 and DCSD-1934300. Boris Muha is partially supported by  the Croatian Science Foundation (Hrvatska Zaklada za Znanost) grant number IP-2018-01-3706. Abner J. Salgado is partially supported by the National Science Foundation under grant DMS-2111228.

\bibliographystyle{plain}
\bibliography{bibfile}

\end{document}